\documentclass[12pt]{amsart}
\usepackage{amsmath}
\usepackage{amsfonts}
\usepackage{amssymb}
\usepackage[all]{xy}           
\usepackage[normalem]{ulem}

\usepackage{bbding}
\usepackage{txfonts}
\usepackage[shortlabels]{enumitem}
\usepackage{ifpdf}
\ifpdf
 \usepackage[colorlinks,final,backref=page,hyperindex]{hyperref}
\else
 \usepackage[colorlinks,final,backref=page,hyperindex,hypertex]{hyperref}
\fi
\usepackage{tikz}

\usepackage{xspace}

\begin{document}

\newcommand{\nc}{\newcommand}
\newcommand{\delete}[1]{}

\nc{\mlabel}[1]{\label{#1}}  
\nc{\mcite}[1]{\cite{#1}}  
\nc{\mref}[1]{\ref{#1}}  
\nc{\mbibitem}[1]{\bibitem{#1}} 

\delete{
\nc{\mlabel}[1]{\label{#1}  
{\hfill \hspace{1cm}{\bf{{\ }\hfill(#1)}}}}
\nc{\mcite}[1]{\cite{#1}{{\bf{{\ }(#1)}}}}  
\nc{\mref}[1]{\ref{#1}{{\bf{{\ }(#1)}}}}  
\nc{\mbibitem}[1]{\bibitem[\bf #1]{#1}} 
}

\newtheorem{theorem}{Theorem}[section]
\newtheorem{thm}[theorem]{Theorem}
\newtheorem{prop}[theorem]{Proposition}
\newtheorem{lemma}[theorem]{Lemma}
\newtheorem{coro}[theorem]{Corollary}
\newtheorem{cor}[theorem]{Corollary}
\newtheorem{prop-def}[theorem]{Proposition-Definition}
\newtheorem{claim}[theorem]{Claim}
\newtheorem{propprop}{Proposed Proposition}[section]
\newtheorem{conjecture}[theorem]{Conjecture}
\newtheorem{assumption}{Assumption}
\newtheorem{condition}[theorem]{Assumption}
\newtheorem{question}[theorem]{Question}
\theoremstyle{definition}
 \newtheorem{defn}[theorem]{Definition}
\newtheorem{exam}[theorem]{Example}
\newtheorem{remark}[theorem]{Remark}
\newtheorem{ex}[theorem]{Example}
\newtheorem{coex}[theorem]{Counterexample}
\newtheorem{conv}[theorem]{Convention}
\renewcommand{\labelenumi}{{\rm(\alph{enumi})}}
\renewcommand{\theenumi}{\alph{enumi}}
\renewcommand{\labelenumii}{{\rm(\roman{enumii})}}
\renewcommand{\theenumii}{\roman{enumii}}

\nc{\abf}{Algebraic Birkhoff Factorization\xspace}
\nc{\cc}{\mathfrak{C}}
\nc{\dm}{\mathcal{DM}}
\nc{\Ima}{\operatorname{Im}}            
\nc{\Dom}{\operatorname{Dom}}      
\nc{\Diff}{\operatorname{Diff}}           
\nc{\End}{\operatorname{End}}        
\nc{\Id}{\operatorname{Id}}                 
\nc{\Isom}{\operatorname{Isom}}      
\nc{\Ker}{\operatorname{Ker}}           
\nc{\Lin}{\operatorname{Lin}}             
\nc{\Res}{\operatorname{Res}}         
\nc{\spec}{\operatorname{sp}}           
\nc{\supp}{\operatorname{supp}}      
\nc{\Tr}{\operatorname{Tr}}                 
\nc{\Vol}{\operatorname{Vol}}            
\nc{\sign}{\operatorname{sign}}         
\nc{\id}{\operatorname{id}}
\nc{\lin}{\operatorname{lin}}
\nc{\I}{J}
\nc{\cala}{\mathcal{A}}
\nc{\deff}{\mathbb{Q}}
\nc{\ot}{\otimes}
\nc{\bfk}{\mathbf{k}}
\nc{\wvec}[2]{{\scriptsize{\big [ \!\!
    \begin{array}{c} #1 \\ #2 \end{array} \!\! \big ]}}}
\nc{\cals}{{\mathcal S}} \nc{\calt}{{\mathcal T}}
\nc{\mlimits}{}

\nc{\zb}[1]{\textcolor{blue}{ #1}}
\nc{\li}[1]{\textcolor{red}{ #1}}
\nc{\sy}[1]{\textcolor{purple}{  #1}}
\nc{\ssy}[1]{\textcolor{orange}{  #1}}

\nc{\lic}[1]{\textcolor{red}{\tt{Li: #1}}}
\nc{\lit}[2]{\textcolor{red}{#1}{ \textcolor{red}{(\sout{#2})}}}
\nc{\lir}[1]{\textcolor{red}{\tt {Li: #1}}}

\newcommand{\Z}{\mathbb{Z}}
\newcommand{\M}{\mathbb{M}}
\newcommand{\ZZ}{\mathbb{Z}}
\newcommand{\Q}{\mathbb{Q}}               
\newcommand{\QQ}{\mathbb{Q}}               
\newcommand{\R}{\mathbb{R}}               
\newcommand{\RR}{\mathbb{R}}               
\newcommand{\coof}{L}           
\newcommand{\coef}{{\mathbb F}}
\newcommand{\p}{\partial}         
\nc{\cone}[1]{\langle #1\rangle}
\nc{\cl}{c}                 
\nc{\op}{o}                 
\nc{\ccone}[1]{\langle #1\rangle^\cl}
\nc{\ocone}[1]{\langle #1\rangle^\op}
\nc{\gc}{\mathcal{C}}
 \nc{\dcc}{\mathfrak{DC}}    
\nc{\dsmc}{\dcc }  
\nc{\csup}{{^\ast}}
\nc{\bs}{\check{S}\,}
\nc{\decc}{{\underline{C}}}
\nc{\ola}[1]{\stackrel{#1}{\longrightarrow}}

 \nc {\linf}{{\rm lin} (F)^\perp}
\nc{\dirlim}{\displaystyle{\lim_{\longrightarrow}}\,}
\nc{\coalg}{\mathbf{C}}
\nc{\barot}{{\otimes}}

\newcommand{\one}{\mbox{$1 \hspace{-1.0mm} {\bf l}$}}
\newcommand{\A}{\mathcal{A}}              
\newcommand{\Abb}{\mathbb{A}}
\renewcommand{\a}{\alpha}                    
\renewcommand{\b}{\beta}                       

\newcommand{\B}{\mathcal{B}}              
\newcommand{\C}{\mathbb{C}}
 \newcommand{\calm}{{\mathcal M}}
\newcommand{\frakf}{{\mathfrak f}}
\newcommand{\CC}{\mathcal{C}}           
\newcommand{\CR}{\mathcal{R}}           
\newcommand{\D}{\mathbb{D}}               
\newcommand{\del}{\partial}                    
\newcommand{\DD}{\mathcal{D}}           
\newcommand{\Dslash}{{D\mkern-11.5mu/\,}} 
\newcommand{\e}{\varepsilon}            
\newcommand{\F}{\mathbb{F}}                
\newcommand{\Ga}{\Gamma}                  
\newcommand{\ga}{\gamma}                   
\renewcommand{\H}{\mathcal{H}}           
\newcommand{\half}{{\mathchoice{\thalf}{\thalf}{\shalf}{\shalf}}}
\newcommand{\hideqed}{\renewcommand{\qed}{}} 
\newcommand{\K}{\mathcal{K}}             
\renewcommand{\L}{\mathcal{L}}          
\newcommand{\la}{\lambda}                   
\newcommand{\<}{\langle}
\renewcommand{\>}{\rangle}
\newcommand{\Mop}{\star}                     
\newcommand{\N}{\mathbb{N}}             
\newcommand{\norm}[1]{\left\lVert#1\right\rVert}    
\newcommand{\norminf}[1]{\left\lVert#1\right\rVert_\infty} 
\newcommand{\om}{\omega}                 
\newcommand{\Om}{\Omega}                
\newcommand{\ol}{\\widetilde}                  
\newcommand{\OO}{\mathcal{O}}          
\newcommand{\ovc}[1]{\overset{\circ}{#1}}
\newcommand{\ox}{\otimes}                    
\newcommand{\pa}{\partial}
\newcommand{\piso}[1]{\lfloor#1\rfloor} 

\newcommand{\rad}{{\mathbf r}}
\newcommand{\sepword}[1]{\quad\mbox{#1}\quad} 
\newcommand{\set}[1]{\{\,#1\,\}}               
\newcommand{\shalf}{{\scriptstyle\frac{1}{2}}} 
\newcommand{\slim}{\mathop{\mathrm{s\mbox{-}lim}}} 
\renewcommand{\SS}{\mathcal{S}}        
\newcommand{\Sp}{{\rm Sp}}
\newcommand{\sg}{\sigma}                              
\newcommand{\T}{\mathbb{T}}                
\newcommand{\tG}{\widetilde{G}}           
\newcommand{\thalf}{\tfrac{1}{2}}            
\newcommand{\Th}{\Theta}
\renewcommand{\th}{\theta}
\newcommand{\tri}{\Delta}                        
\newcommand{\Trw}{\Tr_\omega}           
\newcommand{\UU}{\mathcal{U}}              
\newcommand{\Afr}{\mathfrak{A}}           
\newcommand{\vf}{\varphi}                       
\newcommand{\x}{\times}                          
\newcommand{\wh}{\widehat}                  
\newcommand{\wt}{\widetilde}                 
\newcommand{\ul}[1]{\underline{#1}}             
\renewcommand{\.}{\cdot}                          
\renewcommand{\:}{\colon}                       
\newcommand{\comment}[1]{\textsf{#1}}

\nc{\calc}{\mathcal{C}}
\nc{\calh}{\mathcal{H}}
\nc{\cali}{\mathcal{I}}
\nc{\frakR}{\mathfrak{R}}
\nc{\frakS}{\mathfrak{S}}
\nc {\vep}{\varepsilon}
\nc {\ltcone}{lattice cone\xspace}
\nc {\lC}{(C, \Lambda _C)}
\nc {\conefamilyc}{{\mathfrak{C}}}
\nc {\conefamilyd}{{\mathfrak{D}}}
\nc {\conefamilye}{{\mathfrak{E}}}
\nc {\conefamilyf}{{\mathfrak{F}}}
\nc {\gname}{polar\xspace}
\nc {\simplefraction}{simple simplicial fraction\xspace}
\nc {\simplefractions}{simple simplicial fractions\xspace}
\nc {\Simplefractions}{Simple simplicial fractions\xspace}
\nc {\multiplefraction}{simplicial fraction\xspace}
\nc {\multiplefractions}{simplicial fractions\xspace}
\nc{\pord}{\mbox{\rm p-ord}}
\nc{\pres}{\mbox{\rm p-res}}
\nc {\hfc}{homogeneous family of cones}
\nc {\hfsc}{homogeneous families of cones}
\nc {\hfcs}{homogeneous family of cones \xspace}
\nc {\hfscs}{homogeneous families of cones \xspace}
\nc{\ppp}{projectively properly positioned \xspace}
\nc {\frS}{\mathfrak{S}}

\title[Conical approach to multivariate Laurent expansions]{A conical approach to Laurent expansions for multivariate  meromorphic germs   with linear poles}

\author{Li Guo}
\address{Department of Mathematics and Computer Science,
         Rutgers University,
         Newark, NJ 07102, USA}
\email{liguo@rutgers.edu}

\author{Sylvie Paycha}
\address{Institute of Mathematics,
University of Potsdam,
Am Neuen Palais 10,
D-14469 Potsdam, Germany\\ On leave from the Universit\'e Blaise Pascal\\
Clermont-Ferrand, France}
\email{paycha@math.uni-potsdam.de}

\author{Bin Zhang}
\address{School of Mathematics, Yangtze Center of Mathematics,
Sichuan University, Chengdu, 610064, P. R. China}
\email{zhangbin@scu.edu.cn}

\date{\today}
\begin{abstract} We use convex polyhedral cones to study a large class of multivariate  meromorphic germs, namely those with linear poles,  which naturally arise  in various contexts in mathematics and physics.
We express such a germ as a sum of a holomorphic germ and a linear combination of special non-holomorphic germs called polar germs.
In analyzing the supporting cones -- cones that reflect  the pole structure of the polar germs -- we obtain a geometric criterion for the  non-holomorphicity of linear combinations of polar germs. This yields the uniqueness of the above sum when required to be supported  on a suitable family of cones and assigns a Laurent expansion to the germ.
Laurent expansions provide various decompositions of such germs and thereby a uniformized proof of known results on decompositions  of rational fractions. These Laurent expansions  also yield new concepts on the space of such germs, all of which are independent of the choice of the specific Laurent expansion. These include a generalization of Jeffrey-Kirwan's residue, a filtered residue and a coproduct in the space of such germs. When applied to exponential sums on rational convex polyhedral cones, the filtered residue yields back exponential integrals.
\end{abstract}

\subjclass{32A20, 32A27, 52A20, 52C07}

\keywords{meromorphic functions, convex cones, Laurent expansions, residues, Jeffrey-Kirwan's residue}

\maketitle

\vspace{-2cm}

\tableofcontents

\setcounter{section}{0}

\section {Introduction}
Our study aims at an extension of the classical Laurent theory in one variable to multivariate meromorphic germs.

The classical Laurent theory assigns to a meromorphic germ at zero in one variable a unique Laurent expansion. So the space $\calm_\C(\C)$  of meromorphic germs at zero (say in the variable $\vep $) splits   into the space $\calm_{\C,+}(\C)$ of holomorphic germs at zero and  the  space $\C [[\vep ^{-1}]]=\oplus_{k\leq -1} \C \e^{k} $  consisting of  the polar part of meromorphic germs:
\begin{equation}\label{eq:Laurentdim1}
\calm_{\C}(\C)=\C [[\vep ^{-1}]] \oplus \calm_{\C,+}(\C),
\end{equation}
the direct sum on the right hand side hosting Laurent expansions.
We note that the direct sum decomposition comes as a consequence of the Laurent expansions.

For multivariate meromorphic germs, one naturally asks the same questions, namely how to
\begin{enumerate}
\item define a ``polar part" as a  linear complement of the space of holomorphic germs.
\item obtain a canonical basis or building blocks of this linear complement from which to assign Laurent expansions to meromorphic germs.
\end{enumerate}

While these problems remain open in general, in this paper we provide an answer for another generalization of the Laurent series to  multivariate meromorphic germs, called multivariate meromorphic germs at zero with linear poles, i.e., whose poles lie on unions of hyperplanes.
Such meromorphic germs   arise in various areas of mathematics and in physics as ``regularized" integrals or sums with linear or conical constraints in the variables, and more specifically  in
\begin{itemize}
\item  perturbative quantum field theory when computing Feynman integrals by means
 of analytic regularization \`a la Speer  \cite{Sp1,Sp2} (see also the recent work  by Dang \cite{De}), where the linear constraints in the integration variables correspond to conservation of momentum;
\item    number theory with multiple zeta functions~\cite{Ho,Za} (see also~\cite{GZ,MP,Ma,Zh}) and their generalizations  such as cyclotomic \cite{Te} and  Witten multiple zeta functions \cite{KMT} - the  conical constraint  {$n_1> \cdots> n_k$} on the variables in the summand used to define multiple zeta functions, is responsible for the linearity of the poles;
\item
 the combinatorics on cones  when evaluating exponential integrals or sums on cones following Berline and Vergne {\cite{BV1}} (see also \cite{GP}) in the context of Euler-Maclaurin formula;
\item  algebraic geometry, in particular with the celebrated  Jeffrey-Kirwan residue \cite{JK1,JK2}, see also \cite{dCP} for a review.
\end{itemize}

In our approach of multivariate meromorphic germs with linear poles, the classical Laurent theory generalizes to a {\em local Laurent theory} in the sense that one can cover the space of meromorphic germs at zero  with linear poles by what we call Laurent subspaces.
To define these  Laurent subspaces, we use the geometry of cones. Laurent subspaces are indexed by  properly positioned families of simplicial cones, a  family of cones being properly positioned when the cones meet along faces and their union does not contain any nontrivial subspace. So a meromorphic germ in a Laurent subspace  has a Laurent expansion \emph{supported} by the corresponding properly positioned family of simplicial cones, called \emph{supporting cones}.
In the framework of certain generalized   subdivisions we call pan-subdivisions, one can build a direct system   of properly positioned families of simplicial cones. Thus the corresponding set of Laurent subspaces inherits a  direct system structure whose direct limit is the whole space of meromorphic germs at zero with linear poles.

We explicit below our geometric approach and at the same time give an outline of the paper.

To distinguish the polar part from  the holomorphic part, we fix an inner product and  use the orthogonal complement to define the key concept of polar germs, which serve as building blocks for the polar part. A polar germ is a non-holomorphic germ represented by a fraction  $\frac {h(\ell_1, \cdots , \ell_m)}{L_1^{s_1}\cdots L_n^{s_n}},$  with $h$ is a holomorphic germ at zero in variables $ \ell_1,\cdots, \ell_m$ orthogonal to
the linear forms $ L_1, \cdots ,L_n$ in the pole with multiplicities $s_1,\cdots, s_n$ (Definition \ref{defn:polargerms}).
 We then decompose a meromorphic germ as a sum of polar germs and a holomorphic germ (Theorem \ref{thm:surj}) by operations on fractions.

However such a decomposition is not unique, as the equation $\frac{1}{L_1L_2}=\frac{1}{L_1(L_1+L_2)}+\frac{1}{L_2(L_1+L_2)}$ indicates.
We use the geometry underlying meromorphic germs to address the uniqueness of the decompositions. To a  polar germ $\frac {h(\ell_1, \cdots , \ell_m)}{L_1^{s_1}\cdots L_n^{s_n}},$ we assign  a supporting cone  (Definition \mref{defn:suppcone}) $\langle L_1,\cdots, L_n\rangle $ generated by  the  vectors  $L_1,\cdots, L_n$.

The notion of supporting cones provides a geometric criterion for the non-holomorphicity, in particular the linear independence, of a sum of polar germs in Theorem \ref{thm:non-holo}. More precisely, if the family of supporting cones of a linear combination of polar germs is ``properly positioned" (Definition~\ref{defn:ppp}), then this linear combination cannot be holomorphic.
Properly positioned supporting cones are essential to assign Laurent expansions  to meromorphic germs (Definition \ref{defn:relLaurent}) with the help of a  surjective "forgetful map" (\ref{eq:varphi}) from the space of formal Laurent expansions to the space of meromorphic germs. We first identify the Laurent subspaces in Proposition~\mref{pp:finj} and build from there the Laurent expansion supported on an appropriate properly positioned family of cones. With the notations of the previous example, $\frac{1}{L_1L_2}$ is the   Laurent expansion supported on the properly positioned family $\{\langle L_1,L_2\rangle\}$ while $\frac{1}{L_1(L_1+L_2)}+\frac{1}{L_2(L_1+L_2)}$ is the   Laurent expansion supported on the properly positioned family $\{\langle L_1,L_1+L_2\rangle, \langle L_2, L_1+L_2\rangle\}$.

As with the one variable case, an immediate consequence of the Laurent expansions is a splitting of the space of meromorphic germs with linear poles into a direct sum of the space of holomorphic ones and the space spanned by the polar germs  (Corollary \mref{cor:DecomF}). This direct sum defines in turn a projection onto the holomorphic part along the polar part, a ``multivariate subtraction",  which is multiplicative on orthogonally variate germs (Corollary \ref{cor:proj}), generalizes the minimal subtraction projection operator for meromorphic germs in one variable. This projection is a key ingredient for the algebraic Birkhoff factorization \cite{GPZ2}.

On the grounds of the Laurent expansions and homogeneity properties of the kernel of the forgetful map (Theorem \ref{thm:kernphi}), we equip the space of meromorphic germs with multiple gradings, given by total orders  of the poles -- called the p-order -- of the polar germs arising in the expansion, the spaces spanned by the supporting cones, as well as the dimensions of the supporting cones (Theorem~\ref{thm:SpaceDec}).
These gradings yield several further applications by means of a  uniformized approach. More precisely,
\begin{enumerate}
\item we generalize the Brion-Vergne decomposition~\cite{BrV} $ R_{\mathbf \Delta}= { G}_{\mathbf \Delta}\oplus { NG}_{\mathbf \Delta}$ of rational germs at zero  with poles lying in unions of  hyperplanes in a given  hyperplane  arrangement $\Delta$. See Corollary \mref{thm:BrionVergne}.
\item we obtain the decomposition~\cite[Theorem 7.3]{BV2} of Berline and Vergne as a consequence of a more refined  decomposition in Theorem \ref{thm:SpaceDec}. In contrast to their approach which applies to a meromorphic germ with a prescribed  set of poles determined by a  given hyperplane arrangement $\Delta$, here we consider the whole class of meromorphic germs at zero with linear poles.
\item through a projection to a suitable components from one of the gradings, we define a generalized Jeffrey-Kirwan residue~\cite{BrV,JK1} valid for all meromorphic germs at zero with linear poles in stead of for $ R_{\mathbf \Delta}$. See Corollary~\mref{co:jkr} and Definition~\mref{de:jkr}.
\end{enumerate}

The Laurent expansions have further  interesting consequences leading to new results. The ``p-order" generalized to   meromorphic germs at zero
gives a filtration of the meromorphic germs which generalizes the filtration by the order of the poles on meromorphic germs at zero in one variable and defines a valuation on the division ring of Laurent series \cite[Example 4.2.2]{E}.
We further introduce a filtered residue, the ``p-residue" (Definition \ref{defn:pres}) which,   for Laurent expansions    $\sum\mlimits_{n=-N}^\infty a_n\, x^n$ in one variable  filtered by the valuation given by the order $N\geq 0$ of the poles, corresponds to $\frac{a_N}{x^N}$.
Composed with the exponential sum $S$~\cite{Ba} on a lattice cone ${\rm \pres}\circ S$, the p-residue turns out to be compatible with subdivisions (Proposition~\ref{thm:ressubd}), as a result of which  (see Corollary \mref{cor:resS}), the p-residue yields back the corresponding exponential integral on the lattice cone, related to the former by the Euler-Maclaurin formula studied in \cite{GPZ2}.

Finally, using Laurent expansions, we define in Section~\ref{sec:coproduct} a  coproduct on the space of meromorphic germs at zero with linear poles,  which  is  closely related to  the coproduct on cones  derived in ~\mcite{GPZ3}. This relation is most relevant in the context of renormalization \`a la Connes and Kreimer~\cite{CK} who regarded a renormalized map as a map defined on a coalgebra and taking values in  meromorphic functions.

Throughout this paper, $\F$ denotes a fixed subfield of $\RR$.

\section{A decomposition of meromorphic germs}
\mlabel{s:PolarGerms}
 In order to show  the existence of a decomposition of the space of  meromorphic germs, we first  introduce the concept of a polar germ which will later serve as the building blocks of the linear complement of holomorphic germs.

\subsection{Polar germs}
\mlabel{ss:Germs}
We begin with some necessary preliminary concepts.

\begin{defn} \mlabel{defn:Spaces}
\begin{enumerate}
\item  A {\bf lattice (vector) space} is a pair $(V, \Lambda _V)$ where $V$ is a finite dimensional real vector space and $\Lambda _V$
is a lattice in $V$, that is, a finitely generated abelian subgroup of $V$ whose $\R$-linear span is $V$;
\item An {\bf $\F$-inner product} on a lattice space $(V,\Lambda_V)$ is an inner product $Q$ on $V$ such that the restriction of $Q$ to $\Lambda_{V}\ot \F \subseteq \Lambda_V\ot \R =V$ and hence to $\Lambda_V$ takes values in $\F$;
\item A lattice space with an $\F$-inner product is called an {\bf $\F$-Euclidean lattice space};
\item A {\bf filtered space} is a real vector space $V$ with a filtration $V_1\subset V_2\subset \cdots $ of finitely dimensional real vector subspaces such that $V=\cup_{k\geq 1} V_k$. Let $j_k:V_k\to V_{k+1}$ denote the inclusion;
\item  A {\bf filtered lattice space} is a filtered space $V=\cup_{k\geq 1} V_k$ with lattices $\Lambda_k:=\Lambda_{V_k}$ of $V_k$ such that $\Lambda _{k+1}|_{V_k}=\Lambda _k, k\geq 1$. Then we denote the filtered lattice space by $(V,\Lambda_V)=\cup_{k\geq 1} (V_k,\Lambda_{V_k})$ where $\Lambda_V=\cup_{k\geq 1} \Lambda_{V_k}$;
\item  An {\bf inner product} $Q$ on a filtered space $V=\cup_{k\geq 1}V_k$ is a sequence of inner products
$$ Q_k(\cdot,\cdot)=(\cdot,\cdot)_k: V_k\ot V_k \to \RR, \quad k\geq 1,$$
that is compatible with the inclusions $j_k, k\geq 1$;
\item  An {\bf $\F$-inner product} on a filtered lattice space $(V,\Lambda_V)$ is an inner product $Q:=\{Q_k\}_{k\geq 1}$ on the filtered space $V=\cup_{k\geq 1}V_k$
such that $Q_k$ is an $\F$-inner product for each $k\geq 1$. A filtered lattice space together with an $\F$-inner product is called a \bf{filtered $\F$-Euclidean lattice space}.
\end{enumerate}
\end{defn}

We now assume that $(V,\Lambda_V)=\cup_{k\geq 1} (V_k,\Lambda_k)$ is a filtered $\F$-Euclidean lattice space. Let $V_k^*:=\mathrm{Hom}(V_k,\RR)$ be the dual space of $V_k$. The $\F$-inner product $Q_k: V_k\ot V_k\to \RR $  induces an isomorphism $Q_k^*:V_k\to V_k^*$. This yields an embedding $V_k^*\hookrightarrow V_{k+1}^*$ induced from $j_k:V_k\to V_{k+1}$.
The direct limit
\begin{equation}\label{eq:Vcircledual}
V^\circledast:= \varinjlim V_{k}^*=\bigcup_{k=0}^{\infty}V_{k}^*\end{equation}
is called the {\bf filtered dual space} of $V=\bigcup_{k=0}^{\infty}V_{k}$. Notice that $V^\circledast$ is a proper subspace of the usual dual space $ V^*$ unless $V$ is finite dimensional.

\begin{defn} Let $\cup_{k\geq 1} (V_k,\Lambda_k)$ be a filtered lattice space.
\begin{enumerate}
\item
A  meromorphic germ $f(\vec \e)$ on $V^*_k\otimes \C$ is said to have {\bf  linear poles at zero with coefficients in $\F$} if there exist vectors $L_1, \cdots, L_n\in \Lambda_{V_k}\ot \F$ (possibly with repetitions) such that $f\,\Pi_{i=1}^n L_i$ is a holomorphic germ at zero whose Taylor
 expansion for coordinates in the dual basis $\{e^* _1, \cdots , e^*_k\}$ of a given (and hence every) basis $\{e _1, \cdots , e_k\}$ of $\Lambda_k$ has coefficients in $\F$.
\item
Let  $\calm_\coef  (V^*_k\otimes \C)$ denote the set  of germs of meromorphic functions on $V^*_k\otimes \C$ with linear poles at zero and with coefficients in $\coef$. It is a linear space over $\coef $.
\item A germ of meromorphic functions of the form $\frac{1}{L_1^{s_1}\cdots L_n^{s_n}}$ with linearly independent vectors $L_1,\cdots,L_n$ in $\Lambda_k\ot \F$ and $s_1,\cdots, s_n\geq 1$ is called a {\bf simplicial fraction}  with coefficients in $\F$ or simplicial $\F$-fraction. Such a fraction is called {\bf simple} if   $s_1=\cdots=s_n=1$.
\end{enumerate}
\mlabel{de:fr}
\end{defn}

Since a set of vectors in $\Lambda _k\otimes \F$ is $\F$-linearly independent if and only if it is $\R$-linearly independent in $V_k$, from now on we just call it linearly independent without specifying the type of coefficients.

Composing with the map $j^*_k:V^*_{k+1} \to V^*_{k}$ dual to $j_k:V_k\to V_{k+1}$, yields the embedding
$$\calm_\coef  (V^*_k\otimes \C)\hookrightarrow \calm_\coef  (V^*_{k+1}\otimes \C),$$
giving rise to the direct limit
$$\calm_\coef  (V^\circledast\otimes \C): = \dirlim \calm_\coef (V^*_k\ot \C)=\bigcup_{k=1}^\infty  \calm_\coef  (V^*_k\otimes \C).
$$

Let $\calm_{\coef ,+}(V^*_k\otimes \C)$ denote the space of  germs of holomorphic functions at zero in $V^*_k\otimes \C$ whose Taylor expansions at zero under the dual basis of a basis of $\Lambda  _k$ have coefficients in $\coef$. We set
$$\calm_{\coef ,+} (V^\circledast\otimes \C ) :=\bigcup_{k=1}^\infty\calm_{\coef ,+} (V^*_k\otimes \C).$$
When $\F=\RR$, we usually drop the subscript $\F$ from  the  notation.

When $V_k$ is taken to be $\RR^k$ and is equipped with its standard lattice $\ZZ^k$, the dual space $V_k^*$ is identified with $\R^k$ equipped with the standard lattice. Then the space $\calm_{\coef ,+}(\C^k)=  \calm_{\coef ,+}(V^*_k\otimes \C) $ corresponds to   the space of  germs of holomorphic functions at zero in $\C ^k$ whose Taylor expansions at zero have coefficients in $\coef$ with respect to the canonical basis of $\RR ^k$.

We next identify a linear complement of $\calm_{\coef,+}(V_k^*\ot \C)$ which is canonical upon fixing an inner product on $V_k$. It is spanned by a class of germs which then can be regarded as purely non-holomorphic germs.
More importantly, they will also serve as the building blocks for our Laurent expansions of meromorphic germs in multiple variables with linear poles. See Section~\mref{sect:decom}. For notational simplicity, we will call them \gname germs.

\begin {defn}
\mlabel{defn:polargerms}
Let $(V, \Lambda_V)= \cup_k (V_k,\Lambda_k)$ be a filtered $\F$-Euclidean lattice space with its $\F$-inner product $Q$. A {\bf \gname germ with $\coef$-coefficients} or simply a {\bf $\coef$-\gname germ} in $V^*_k\otimes \C$ is a germ of meromorphic functions at zero of the form
\begin{equation}\label{eq:polargerm}\frac {h(\ell_1, \cdots , \ell_m)}{L_1^{s_1}\cdots L_n^{s_n}},
\end{equation}
where
\begin{enumerate}
\item
$h$ lies in $\calm_{\coef ,+}(\C^m)$,
\item
$\ell_1, \cdots, \ell_m, L_1, \cdots ,L_n$  lie in $\Lambda _k\otimes \coef $, with $\ell_1,\cdots,\ell_m$ and $L_1, \cdots ,L_n$ linearly independent, such that
$$Q(\ell_i, L_j)=0 \quad \text{ for all } (i,j)\in [m]\times [n],$$
where for a positive integer $k$, we have set $[k]:=\{1,\cdots, k\}$,
\item
$s_1,\cdots, s_n$ are positive integers.
\end{enumerate}
For notational simplicity, we shall also set
$\vec{L}^{\vec s}:=L_1 ^{s_1}\cdots L_n^{s_n}$
and write  \begin{equation}
\label{eq:hL}\frac{h(\vec\ell)}{\vec L^{\vec s}}:= \frac {h(\ell_1, \cdots , \ell_m)}{L_1^{s_1}\cdots L_n^{s_n}}.
\end{equation}
\end{defn}

\begin{defn} We let $\calm_{\coef, -}^Q(V^*_k\otimes \C)$ denote the linear subspace of $\calm _\F(V^*_k\ot \C)$ spanned by $\coef$-\gname germs and set
$$\calm_{\coef , -}^Q  (V^\circledast\otimes \C): =\bigcup_{k=1}^\infty  \calm_{\coef, -}^Q(V^*_k\otimes \C)=\dirlim \calm_{\coef, -}^Q(V^*_k\otimes \C),
$$
regarding $\{\calm_{\coef, -}^Q(V^*_k\otimes \C)\}_k$ as a sub-direct system of $\{\calm_{\coef}(V^*_k\ot \C)\}_k$.
\end{defn}

\begin{remark}
\mlabel {rem:CounterExp}
The space $\calm_{\coef ,-} (V^\circledast\otimes \C )$ is not closed under the function product. As a simple example, for the canonical inner product on $\R^2$, both $f(\e_1,\e_2):=\e_1/\e_2$ and $g(\e_1,\e_2):=\e_2/\e_1$ are polar germs. But their product $1$ is not.
\end{remark}

\begin{ex}
\begin{enumerate}
\item
For linearly  independent vectors $L_1,\cdots ,L_k \in \Lambda_k\ot \coef $ and $ s_1,\cdots,  s_k> 0,$
$\displaystyle{\frac {1}{L_1^{s_1}\cdots L_k^{s_k}}}$ lies in $\calm_{\coef,-}^Q(V^*_k\otimes \C)
$  for any inner product $Q$.
\item
For the canonical Euclidean inner product on $\RR^2$, the functions $f(\e _1 e_1^* +\e _2 e^*_2)=\frac{(\e_1-\e_2)^t}{(\e_1+\e_2)^s},$ $s>0, t\geq 0,$ lie in $\calm_{\Q,-}^Q ((\RR^{2})^*\otimes \C)$.
\end{enumerate}
\end{ex}

\begin {remark} We  will mostly be working with a filtered $\F$-Euclidean lattice space given by a fixed $\F$-inner product $Q$. Thus we often  drop the superscript $Q$ to simplify notations.
\end{remark}

The following lemma shows the uniqueness of the expression of a polar germs.

\begin {lemma}
If a \gname germ can be written as $\frac {h(\ell_1, \cdots , \ell_m)}{L_1^{s_1}\cdots L_n^{s_n}}$ and $\frac {g(\ell ^\prime _1, \cdots , \ell ^\prime _j)}{(L _1^\prime)^ {t_1}\cdots (L_\ell^\prime)^{t_\ell}}$, both in a form satisfying the conditions in Definition~\mref{defn:polargerms}, then $n=\ell$ and   $L_1^\prime, \cdots, L_\ell^{\prime}$ can be rearranged in such a way that $L_i$ is a multiple of $L^\prime_i$ and $s_i=t_i$ for $1\leq i\leq n$.
\mlabel{lem:supp}
\end {lemma}
\begin {proof} We implement an   induction on $M:=\max(s_1+\cdots +s_n, t_1+\cdots +t_\ell)$.

To deal with the case when $M=1$, assume $ \frac {h(\ell _1, \cdots, \ell _m )}{L }$ equals $\frac {g(\ell ^\prime _1, \cdots , \ell ^\prime _j)}{L ^\prime}$ or $g(\ell ^\prime _1, \cdots , \ell ^\prime _j)$, with $h$ and $g$ holomorphic. Extend $\{L, \ell _1, \cdots , \ell _m\}$ to a basis $\{ z_1, \cdots, z_{m+1}, \cdots , z_k\}$, with $z_1=L, z_2=\ell _1, \cdots, z_{m+1}=\ell _m$.
If $ \frac {h(\ell _1, \cdots, \ell _m )}{L }=\frac {g(\ell ^\prime _1, \cdots , \ell ^\prime _j)}{L ^\prime}$, and $L'$ is not a multiple of $L$, then
$$L'(z_1,z_2, \cdots , z_k)=L^{\prime \prime}(z_2, \cdots , z_k)+c z_1,$$
with $L^{\prime \prime}(z_2, \cdots , z_k)$ not identically zero, so we can pick $z_2^0, \cdots z_{k}^0$, such that $h(z_2^0, \cdots, z_{m+1}^0)\not=0$ and $L^{\prime \prime}(z^0_2, \cdots , z^0_k)\not =0$. Consider the restriction to $(z_1, z_2^0, \cdots, z_{k}^0)$ of the equality $ \frac {h(\ell _1, \cdots, \ell _m )}{L }=\frac {g(\ell ^\prime _1, \cdots , \ell ^\prime _j)}{L ^\prime}$.
The left hand side of the equality is singular in $z_1$ while the right hand side is holomorphic in $z_1$. This is a contradiction, showing that $L$ must be a multiple of $L'$.
The same argument shows that $ \frac {h(\ell _1, \cdots, \ell _m )}{L }=g(\ell ^\prime _1, \cdots , \ell ^\prime _j)$ is impossible.

For the inductive step,
suppose that none of the linear forms $L _1^{\prime}, \cdots, L_\ell^{\prime}$  is a multiple of $L_1$. As in the case for $M=1$, when
$$\frac {h(\ell_1, \cdots, \ell_m)}{L_1^{s_1}}=\frac {g(\ell ^\prime _1, \cdots , \ell ^\prime _k)L_2^{s_2}\cdots L_n^{s_n}}{(L _1^{\prime})^ {t_1}\cdots (L_\ell^{\prime})^ {t_\ell}}$$
is restricted to a proper choice of $z_2^0, \cdots, z_{k}^0$, the left hand side of the equality has a non-trivial singular part in $z_1$, while the right hand side is holomorphic in $z_1$, which  leads to a contradiction. Therefore we can rearrange $L_1^{\prime},\cdots, L_\ell^{\prime}$ so that
$L_1=cL^\prime _1 $
for some constant $c\neq 0$. Thus from
$$\frac {h(\ell_1, \cdots , \ell_m)}{L_1^{s_1}\cdots L_n^{s_n}}=
\frac {g(\ell ^\prime _1, \cdots , \ell ^\prime _j)}{(L _1^\prime)^ {t_1}\cdots (L_\ell^\prime)^ {t_\ell}},$$
we obtain
$$\frac {h(\ell_1, \cdots , \ell_m)}{L_1^{s_1-1}\cdots L_n^{s_n}}=
\frac {cg(\ell ^\prime _1, \cdots , \ell ^\prime _j)}{(L _1^\prime)^ {t_1-1}\cdots (L_\ell^\prime)^ {t_\ell}}.$$
By the inductive hypothesis, the conclusion holds for the two sides of the equation. This completes the induction.
\end{proof}

\subsection{Decomposition of a meromorphic germ into polar germs}
In this subsection we consider any lattice space $(V,\Lambda)$ which can be taken to be $(V_k,\Lambda_k)$ from a filtered lattice space. The notions for $V_k$ such as $\calm_\F(V_k^\ast\ot \C)$ and $\calm_{\coef,\pm}(V_k^\ast\ot \C)$ can be defined in the same way for $V$.

Before giving the decomposition, we first provide some preliminary results.

\begin{lemma} Let $(V,\Lambda)$ be a lattice space. Let $L_1,\cdots,L_n, n\geq 2,$ be  vectors in $\Lambda\ot \F$ and let $s_1,\cdots, s_n$  be positive integers.
\begin{enumerate}
\item  If  $L_1,\cdots,L_n$ are $\F$-linearly independent and $L_{n+1}=\sum_{i=1}^n c_i L_i$ with nonzero $c_i\in \F, 1\leq i\leq n$. Then
\begin{equation}
\frac{1}{L_1^{s_1}\cdots L_{n+1}^{s_{n+1}}}= \sum_j \frac {b_j}{N^{t_{j1}}_{j1}\cdots N^{t_{jn}}_{jn}},
\mlabel{eq:frind2}
\end{equation}
where, for each $j$, $b_j$ is in $\coef$ and $\{N_{j1},\cdots,N_{jn}\}$ is one of the sets $\{L_1,\cdots,\widehat{L_i},\cdots,L_{n+1}\}, 1\leq i\leq n$ (where $\widehat{L}_i$ means that the factor $L_i$ is omitted) and hence is a basis of the linear span $\lin \{L_1,\cdots,L_{n+1}\}$ of $L_1,\cdots,L_{n+1}$.
\mlabel{it:dec1}
\item  In general,  the fraction $\frac{ 1}{L_1^{s_1}\cdots L_n^{s_n}}$ can be rewritten as a linear combination
 $$\sum_i \frac {a_i}{M^{t_{i1}}_{i1}\cdots M^{t_{in_i}}_{in_i}},
$$
with $a_i\in \coef $ and linear independent subsets $\{M_{i1}, \cdots, M_{in_i}\}$ of $\{L_1,\cdots,L_n\}$.
\mlabel{it:dec2}
\end{enumerate}
\mlabel{lem:merodec}
\end{lemma}

\begin{proof}
(\mref{it:dec1})
The statement easily follows from  the straightforward identity
$$\frac 1{L_1\cdots L_{r}}=\sum_{i=1}^r\frac {c_i}{L _1\cdots \widehat{L}_i\cdots L_rL_{r+1}},
$$ by induction on the sum $m:=\sum_{j=1}^{r} s_j$.
\smallskip

\noindent
(\mref{it:dec2}) Combining factors of linear forms that are multiples of each other if necessary, we can assume that  the $L_i$'s are not  multiples of each other. The statement then follows from an induction on the difference $d:=n-\dim ( \lin\{L_1,\cdots,L_n\})$ using  Eq.~(\mref{eq:frind2}) applied to a subset $L_{i_1},\cdots,L_{i_r}$  of  linearly independent forms such that $L_{i_{r}+1}=\sum_{j=1}^r c_j L_{i_j}$ for some $2\leq r\leq n$.
\end{proof}

We are now ready to prove the existence of a decomposition of meromorphic germs at zero into a sum of holomorphic germs and polar germs.

\begin{thm}
\mlabel{thm:surj}  Let $(V, \Lambda)$ be an $\F$-Euclidean lattice space with an $\F$-inner product $\Q$.
For any $f\in \calm_{\coef}(V^*\otimes \C)$,  there exists a finite set of $\F$-polar germs  $\{S_j\}_{j\in J}$ and a holomorphic germ $h$ in $\calm_{\coef}(V^*\otimes \C)$ such that
\begin{equation}\label{eq:existencedec}f=\left(  \sum_{j\in J} S_j\right)+\, h.
\end{equation}
Furthermore, the $\F$-polar germs $S_j$ can be chosen to satisfy the following properties.
\begin {itemize}
\item their linear poles are taken from the linear poles of $f$.
\item
if the germ $f$ can be written in the form $\tilde f(\ell_1, \cdots, \ell_n)$ for some function $\tilde f$ on $\C ^n$ and linearly independent linear forms $\ell_1,\cdots, \ell_n$ on $(\Lambda\ot \F)^*$, then the polar germs $S_j$, and the holomorphic germ $h$ can be written as compositions of functions on $\C ^n$ and linearly independent linear forms in ${\rm span}( \ell_1,\cdots, \ell_n)$.
\end{itemize}
\end{thm}

\begin{remark} Whereas the holomorphic part will turn out to be uniquely defined, the individual polar germs   arising in this decomposition are not. In Corollary~\mref{coro:uniqueness}
we provide geometric conditions under which the polar germs can be uniquely determined, leading to Laurent expansions.
\end{remark}

\begin{proof} Thanks to Lemma~\mref{lem:merodec}.(\mref{it:dec2}), without loss of generality we can reduce the proof to meromorphic germs at zero of the form
$$f=\frac {h}{L^{s_1}_{1}\cdots L_{m}^{s_m}}
$$
with $h\in \calm  _{\F,+}(V^*\otimes \C)$, $L_1,\cdots,L_m\in \Lambda \otimes \coef$ linearly independent  and $s_1, \cdots, s_m$ positive integers. Then we
extend $\{L_1, \cdots , L_m\}$ to a basis $\{L_1, \cdots , L_m, \ell_1, \cdots , \ell_{k-m}\}$ of $\Lambda \otimes \F$ satisfying
$$Q(L_i,\ell_j)=0, \quad 1\leq i\leq m, 1\leq j\leq k-m.
$$

We proceed by induction on the sum $s:=s_1+\cdots+s_m$. If $s=1$, then $m=1$ and $s_1=1$.
Under these conditions we have
$$f=\frac{h(L_1, \ell _1, \cdots , \ell _{k-1} )}{L_{1}}= \frac {h(0, \ell _1, \cdots , \ell _{k-1})}{L_1}+\frac {h(L_1, \ell _1, \cdots , \ell _{k-1})-h(0, \ell _1, \cdots , \ell _{k-1})}{L_1}.$$  The first term
  lies  in $\calm _{\F,-}(V^*\otimes \C)$ as a consequence of the orthogonality of $L_1$ with the $\ell_i$'s.
 The second term is holomorphic at $0$. This yields the required decomposition.

For $t\geq 1$, assume that the decomposition exists when $s\le t$ and consider  $f=\frac {h(L_1,\cdots,L_m,\ell_1,\cdots,\ell_{k-m})}{L^{s_1}_{1}\cdots L_{m}^{s_m}}$ with $s=s_1+\cdots +s_m=t+1$. We note that~\cite{GF} the Taylor expansion of $h$ gives
$$h(L_1,\cdots,L_m,\ell_1,\cdots,\ell_{k-m}) :=h(0,\cdots,0,\ell_1,\cdots,\ell_{k-m}) +\sum_{i=1}^m L_i g_i,$$
where the $g_i$'s are holomorphic.
Thus $S_0:=h(0,\cdots,0,\ell_1,\cdots,\ell_{k-m})/(L^{s_1}_{1}\cdots L_{m}^{s_m})$ is in  $\calm_{\F,-}(V^\ast\otimes \C)$ while by the induction hypothesis,
$(L_ig_i)/(L^{s_1}_{1}\cdots L_{m}^{s_m})=h_i+\sum_{j_i\in J_i} S_{j_i} $ with $h_i$ a holomorphic germ at zero and $S_{j_i}$ polar germs in $\calm_\coef(V^\ast \ot \C)$.
 Hence
$f=S_0+\sum_{i=1}^m h_i+\sum_{ j\in \cup_{i=1}^mJ_i} S_{j}$ is the sum of a holomorphic germ $\sum_{i=1}^m h_i$ and finitely many polar germs $S_j$.

Now for a germ $f$ expressed in the form $\tilde{f}(\ell_1,\cdots,\ell_n)$ as given in the theorem, replace the lattice space $(V,\Lambda)$ by its lattice subspace $(W,\Lambda\cap W)$ where $W:=\text{span}(\ell_1,\cdots,\ell_n)$. Then $f$ is in $\calm_\coef(W^\ast\ot \CC)$ and applying the first part of the theorem yields the second part of the theorem.
\end{proof}

\section{A geometric criterion for non-holomorphicity}
\mlabel{sec:nonholo}
In this section, we pursue our geometric approach initiated in \cite{GPZ1} to study meromorphic germs at zero through the cones associated to the germs.
By means of the  supporting cone of a polar germ, we first give a geometric criterion for the linear independence of simplicial fractions in Section~\mref{ss:geo1}. We then obtain the main Non-holomorphicity Theorem in Section~\mref{ss:geo2}.

\subsection{A geometric criterion for the linear independence of simplicial fractions}
\mlabel{ss:geo1}
We briefly recall the notations  and terminology of ~\cite{GPZ1} and  use the results  obtained there on the geometry of cones underlying the decomposition of fractions, further refined to require  that the coefficients  lie in the subfield $\F$.

As in ~\cite{GPZ1}  we consider  {\bf closed convex polyhedral cones}  henceforth simply called  {\bf cones} in a filtered lattice space $(V,\Lambda_V)=\cup_{k\geq 1} (V_k,\Lambda_{V_k})$.
 We call {\bf $\F$-cones} the ones whose generators lie   in $\Lambda _k \otimes \F$. A $ \QQ $-cone  is called {\bf rational}.

We recall that a {\bf subdivision} of a cone $C$ is a set $\{C_1,\cdots,C_r\}$ of cones which have the same dimension as $C$,  whose union is $C$,  and that intersect along their faces, i.e.,  $C_i\cap C_j$ is a face of both $C_i$ and $C_j$.
Such a subdivision is  {\bf simplicial} (resp. {\bf smooth}, in the case when $C$ is rational) if all $C_i$'s are simplicial (resp. smooth). An {\bf $\F$-subdivision} of an $\F$-cone is a subdivision such that every $C_i$ is an $\F$-cone.

On the grounds of Lemma~\mref{lem:supp}, we can assign a simplicial cone to a \gname germ.

\begin {defn}\mlabel{defn:suppcone}
Let
$$f:=\frac {h(\ell_1, \cdots , \ell_m)}{L_1^{s_1}\cdots L_n^{s_n}},
$$
be a polar germ, as defined by the conditions in Definition~\mref{defn:polargerms}. We  shall say that the cone $\cone{L_1, \cdots , L_n}$  {\bf supports}  the germ; it is a {\bf supporting cone of the germ}.
\end{defn}

By Lemma~\mref{lem:supp}, the supporting cones of a polar germ is defined up to the choice of a sign of each of the vectors $L_1,\cdots,L_n$. Indeed, any cone $\cone{\pm L_1, \cdots ,\pm L_n}$  delimited by the hyperplanes  in the hyperplane arrangement  $\{H_1,\dots, H_n\}$ with $H_i=\{L_i=0\}$ is also a supporting cone.
For example, in the standard Euclidean space $\RR^2$, the polar germ $\frac{1}{\e_1 \e_2}$ has four supporting cones given by the four quadrants of the Euclidean plane, cut out by the two lines spanned by  the basis vectors $e_1$ and $e_2$ respectively.

We now introduce the key concepts concerning families of cones.

\begin {defn}
\begin{enumerate}
\item
A family of cones is said to  be  {\bf   properly positioned} if the cones meet along faces and the union does not contain any nonzero linear subspace.
\item
A family of polar germs  is called {\bf properly positioned} if there is a choice of a supporting cone for each of the polar germs such that the resulting family of cones is properly positioned.
\item
A family of polar germs   is called {\bf \ppp} if it is properly positioned and none of the denominators of the polar germs is proportional to another.
\end{enumerate}
\label{defn:ppp}
\end{defn}

So a family of polar germs   is \ppp if it is properly positioned when viewed in the projective space of polar germs (that is,  modulo scalar multiples), hence the terminology.
It would be interesting to find a simple criterion for a \ppp family of polar germs.

We next give a reinterpretation of a related result in~\cite{GPZ1} for which we recall some preliminary notations and results. See also~\cite {BV1,GP,La}.

Let $C$ be a simplicial cone in $V_k$ with $\RR$-linearly independent generators $v_1, \cdots v_n$ expressed in a fixed basis $\{e_1,\cdots,e_k\}$  as $v_i=\sum\mlimits_{j=1}^k a_{ji}e_j $, for $1\leq i\leq n$.  Define linear functions $L_i, 1\leq i \leq n,$ on $V_k^*\ot \C$ by $L_i(\vec \e):=L_{v_i}(\vec \e):=\sum\mlimits_{j=1}^k a_{ji}\e _j$, where $\vec \e:=\sum\mlimits_{i=1}^k \e_je_j^* \in V_k^*\otimes \C $ and $\{e_1^*, \cdots, e_k^* \}$ is the dual basis in $V_k^*$. Let $A_C=[a_{ij}]$ denote the associated matrix in $M_{k\times n}(\RR)$.
Let $w(v_1, \cdots,
v_n)$ or $w(C)$ denote the sum of absolute values of the determinants of all minors of $A_C$ of rank $n$.
As in~\mcite{GPZ1} except a different notation $\Phi$ instead of $I$ and a sign convention, define
\begin{equation}
I (C):=(-1)^n\frac
{w(v_1, \cdots , v_n)}{L_1\cdots L_n}.
\mlabel{eq:phi0}
\end{equation}

Let $C$ be a cone in $V_k$ and let $\{C_i\}$ be a simplicial subdivision of $C$. By \cite[Lemma~3.3]{GPZ1}, the sum
\begin{equation}
 I(C):=\sum_i I(C_i)
\mlabel{eq:phi1}
\end{equation}
is well-defined, independent of the choice of simplicial subdivisions, hence yielding a linear map
\begin{equation}\mlabel{eq:Int}
I: \R\gc (\RR )\to \calm_{\coef, -} (V^\circledast \otimes \C),
\end{equation}
 where $\R\gc (\RR )$ is the $\R $-linear space spanned by the set $\gc (\RR )$ of cones in $V$.

Now we are ready for our first geometric criterion for the linear independence of fractions.

\begin{lemma}
A \ppp family of simple  $\F$-fractions whose supporting cones span the same linear subspace is $\F$-linearly independent.
\mlabel{lem:gencone2}
\end{lemma}

\begin{proof}
We choose the supporting cones $\{C_i\}$ in such a way that the family is properly positioned. Since the simple $\F$-fractions are not pairwise proportional, these supporting cones are distinct. Since the cones are properly positioned, their union does not contain any nonzero linear subspace. Thus the union of the cones has a topological boundary. Thus by~\cite[Lemma~3.5]{GPZ1}, the set $\{I(C_i)\}$ is linearly independent. But each $I(C_i)$ is a nonzero multiple of the original fraction. Thus the original family of simple fractions is linearly independent.
\end{proof}

Before the treatment of more general fractions, we give the following ``locality" lemma.

\begin{lemma}
Let $\frac {h_i}{ \vec L_i^{\vec s_i}}, \ i=1,\cdots, r$
be $\F$-polar germs and $h_0$ a holomorphic germ at zero satisfying \begin{equation}
\mlabel{eq:ce}
\sum\mlimits_{i=1}^r a_i\frac {h_i}{\vec L_i^{\vec s_i}}=h_0
\end{equation}
with $a_1,\cdots, a_r\in \F$.
For any linear $\F$-subspace $W$ of $V$ and $N \in \ZZ _{>0}$, denote
$$I(W,N):=\{i\in[r]\, |\, {\rm span}(L_{i1}, \cdots, L_{in_i})=W, \vert s_i\vert: =s_{i1}+\cdots +s_{in_i}=N \}.
$$
Then
$$\sum\mlimits_{i\in I(W,N)} a_i\frac {h_i}{\vec L_i^{\vec s_i}}=0,$$
with the convention that the sum over an empty set is zero.
\mlabel{lem:SepSpace}
\end{lemma}

\begin {proof} For  distinct pairs $(W,N)$ and $(W',N')$ arising in the expression Eq.~(\mref{eq:ce}) we have $I(W,N)\cap I(W',N')=\emptyset$. Thus
$[r]$ is partitioned into finitely many non-empty and disjoint subsets $I(W_1, N_1)$, $\cdots$, $I(W_p, N_p)$.
Then
$$\sum _{j=1}^p\sum _{i\in I(W_j,N_j)}a_i\frac {h_i}{\vec L_i^{\vec s_i}}=\sum _{i=1}^ra_i\frac {h_i}{\vec L_i^{\vec s_i}}=h_0.
$$

Suppose that an expression in Eq.~(\mref{eq:ce}) is a counter example to the lemma. Then
\begin{equation}
\sum _{i\in I(W_j,N_j)}a_i\frac {h_i}{\vec L_i^{\vec s_i}}\neq 0
\mlabel{eq:nzero}
\end{equation}
for some $j\in[p].$
By dropping those $j\in [p]$ with $\sum _{i\in I(W_j,N_j)}a_i\frac {h_i}{\vec L_i^{\vec s_i}}= 0$ if necessary, we can assume that Eq.~(\mref{eq:nzero}) holds for all $j\in[p].$

Let
$N:= \max\{ \vert s_i\vert\,|\, i\in [r]\}$ and let $W$ be one of those ${\rm span}(L_{i1}, \cdots, L_{in_i})$ with $\vert s_i\vert=N$ whose dimension is minimal.
Reordering the terms of  the sum in Eq.~(\mref{eq:ce}) if necessary, we can assume that $I(W,N)=[t]$ for some $t\geq 1$. Thus $|s_i|=N$ and ${\rm span}(L_{i1}, \cdots, L_{in_i})=W$ precisely for $i\in [t]$.

We extend the linearly independent linear forms $L_{11},\cdots,L_{1n_1}$ to a basis $e_1,\cdots,e_k$ of $\Lambda_k\ot \F$ with $e_i=L_{1i}$ for $i\in [n_1]$ such that $Q(e_j,e_\ell)=0$ for $1\leq i\leq n_1, n_1+1\leq j\leq k$. Write the polar germs $\frac{h_i}{\vec{L_i}^{\vec{s_i}}} =\frac{h_i(\ell_{i1},\cdot,\ell_{im_i})}{L_{i1}^{s_{i1}}\cdots L_{in_i}^{s_{in_i}}}$ as in Definition~\mref{defn:polargerms}. Since $e_1,\cdots,e_{n_1}$ is also a basis of ${\rm span}(L_{i1}, \cdots, L_{in_i})$ for $i\in [t]$, we have
$$Q(e_j,\ell_{ij})=0, 1\leq j\leq n_i, 1\leq \ell\leq m_i, 1\leq i\leq t.$$
So the linear forms $\ell_{i1},\cdots,\ell_{im_i}$ lie in ${\rm span}(e_{n_1+1},\cdots,e_k)$.
Thus with respect to the dual basis $\{e_1^*, \cdots, e^*_k\}$ of the basis $\{e_1,\cdots,e_k\}$ of $V_k\otimes \C$, the functions $h_1 (\ell_{1\,1},\cdots,\ell_{1\,m_1}), \cdots, h_{t}(\ell_{t\,1},\cdots,\ell_{t\,m_t})$ as functions in the variables $\vec \e =\sum \e _i e_i^*$ are in fact germs at zero of holomorphic functions depending only on the variables $\e_{n_1+1}, \cdots \e_k$, which we write as ${h}_1(\e_{n_1+1},\cdots,\e_k),\cdots,{h}_t(\e_{n_1+1},\cdots,\e_k)$.

Fix $i>t$. For any $j\in [n_i]$, write
$L_{ij}=L_{ij}'+L_{ij}^{\prime \prime},$
where $L_{ij}'$ is a linear combination of $e_1, \cdots, e_{n_1}$ and $L_{ij}^{\prime \prime}$ is a linear combinations of $e_{n_1+1}, \cdots e_k$. Thus $L_{ij}^{\prime\prime}(\vec \e)$ is a linear function in $\e_{n_1+1},\cdots,\e_k$. We note that $i>t$ if and only if either $\sum_{j=1}^{n_i} s_{ij}<N$, or there is an index  $j$ such that $L_{ij}^{\prime\prime}\neq 0$ as a result of the fact that $\{L_{i1},\cdots,L_{in_i}\}$ and $\{L_{11},\cdots,L_{1n_1}\}$  do not span the same linear space.

Since $h_i(\e_{n_1+1} , \cdots, \e_k ) , 1\leq i\leq t,$ are not identically zero, there are fixed values  $\e_{n_1+1}^0, \cdots, \e_k^0$ of $\e_{n_1+1}, \cdots, \e_k$ for which $h_i(\e_{n_1+1}^0, \cdots, \e_k^0)\not =0, 1\leq i\leq t,$  and   $L_{ij}^{\prime \prime}(\e_{n_1+1}^0, \cdots \e_k^0)\not =0, i>t$ for those $L_{ij}^{\prime\prime}\neq 0$. These values form a non-empty open subset.

We next introduce a new set of variables $r_m, 1\leq m\leq n_1,$ and $\varepsilon$, and apply the substitution $(\e_1,\cdots,\e_k) =(r_1\varepsilon,\cdots,r_{n_1}\varepsilon,$ $\e_{n_1+1}^0,\cdots,\e_k^0)$
in Eq.~(\mref{eq:ce}). This gives rise to a Laurent series in $\varepsilon$ that is holomorphic at zero  by the choice of the germ. Thus the coefficient of every given negative power of $\varepsilon$ is $0$. In particular the coefficient of the least possible power $\varepsilon ^{-N}$ is zero.
In order for a term $h_i/L_{i1}^{s_{i1}}\cdots L_{in_i}^{s_{in_i}}$ in the sum to contribute to this coefficient, we must have $\sum_{j} s_{ij}=N$ and $L^{\prime\prime}_{ij}=0$, that is, $1\leq i\leq t$ as a result of the definition of $t$. On the other hand, for $1\leq i\leq t$, $L_{i1},\cdots,L_{in_i}$ are linear homogeneous in $L_{11}(\vec \e)=\e_1,\cdots,L_{1n_1}(\vec \e)=\e_{n_1}$. Hence under the above substitution, they give $\e L_{i1},\cdots,\e L_{in_i}$ in the variables $r_1,\cdots,r_{n_1}$ and  the coefficient of $\e^{-N}$ in the Laurent series reads
$$\sum_{i=1}^t a_i\frac
{h_i(\e_{n_1+1}^0,\cdots,\e_k^0)}{L_{i1}^{s_{i1}}\cdots L_{in_i}^{s_{in_i}}}.$$
Hence it is zero as a sum of fractions in variables $r_1,\cdots,r_{n_1}$.
Thus
\begin{equation}\label{eq:sumtzero}\sum_{i=1}^t a_i\frac
{h_i(\e_{n_1+1}^0,\cdots,\e_k^0)}{L_{i1}^{s_{i1}}\cdots L_{in_i}^{s_{in_i}}}=0
\end{equation}
for any generic point $(\e_{n_1+1}^0,\cdots,\e_k^0)$. Comparing with Eq.~(\mref{eq:nzero}) gives the desired contradiction.
\end{proof}

Based on this lemma, Lemma~\mref {lem:gencone2} can be generalized to the following statement.

\begin{prop}
A \ppp family of simplicial  $\F$-fractions is $\F$-linearly independent.
\mlabel{lem:LinearlyIndep}
\end{prop}

\begin{proof}
We only need to prove that a contradiction follows from any
linear relation
\begin{equation}
\sum_{i=1} ^r a_i \frac{1}{L_{i1}^{s_{i1}}\cdots L_{in_i}^{s_{in_i}}}=0, \quad 0\neq a_i\in \coef,
\mlabel{eq:contrel}
\end{equation}
of a \ppp family of $\F$-fractions
$G_i:=\frac{1}{L_{i1}^{s_{i1}}\cdots L_{i n_i}^{s_{in_i}}}$.

By Lemma~\mref {lem:SepSpace}, we can assume that, for each $1\leq i\leq r$, the weight $\vert s_i\vert =s_{i1}+\cdots +s_{in_i}$ is the same and the linear forms in the denominators span the same space. In particular, $n_1=\cdots =n_r=n$.

We next proceed by induction on $s:=\vert s_i \vert$. So $s\geq n$.
If $s=n$, then the powers of all the linear forms are equal to $1$. It then follows from  Lemma~\mref {lem:gencone2} that  $a_i=0$ for all indices $i$, leading to the expected contradiction. Assume that a contradiction arises for any relation in Eq.~(\mref{eq:contrel}) with $s= N\geq k$ and consider such a relation with $s=N+1$. In this case, at least one linear form, say $L_{1}$, has exponent greater than one.

Let $r_1$ be the maximal power of $L_{1}$ in all the \multiplefractions $G_i, 1\leq i \leq r$. We split these fractions into three disjoint sets. Let $G_1, \cdots, G_m$ be all the \multiplefractions with $L_{1}$ raised to the power of $r_1$. Let $G_{m+1},\cdots,G_{m+\ell}$ be all the \multiplefractions, if any, with $L_{1}$ raised a positive power less than $r_1$. Let $G_{m+\ell+1},\cdots,G_r$ be all the \multiplefractions, if any, that do not contain $L_1$ in their denominator. Thus
$$0=L_1\sum_{i=1}^r a_i   G_i= \sum_{i=1}^m a_i L_{1} G_i + \sum _{i=m+1}^{m+\ell}  a_iL_{1}  G_i + \sum _{i=m+\ell+1}^r  a_i L_{1}G_i.
$$
For any $m+1\leq i\leq m+\ell$, the power of $1/L_{1}$ in $L_1G_i$ is less than $r_1-1$.  Since the linear forms in the denominators are assumed to span the same spaces, we  write $L_{1}$ as a linear combination of the linear forms $L_{i1},\cdots,L_{in_i}$ of $G_i$ for $m+\ell +1\le i \le r$:
$$L_{1}=a_{i1}L_{ i1}+\cdots +a_{in_i}L_{in_i}.
$$
Thus each $L_{1}G_i=\sum\mlimits_{j_i=1}^k \frac{a_{ij_i}}{L_{i1}^{s_{i1}}\cdots L_{ij_i}^{s_{ij_i}-1}\cdots L_{in_i}^{s_{in_i}}}$ for $m+\ell+1\leq i\leq r$ is a linear combination of fractions that do not contain $L_1$  as a linear form in the denominator. In summary, each fraction in $\sum\mlimits_{i=m+1}^r a_iL_1 G_i $ has its power of $1/L_{1}$ less than $r_1-1$, so that no such monomial can  cancel  with any fraction in $\sum\mlimits_{i=1}^m a_iL_1 G_i$.

With the notation of Lemma~\mref {lem:SepSpace}, in
$\sum\mlimits_{i=1}^r a_iL_{1}G_i=0$
we can use the space spanned by $L_{11}, \cdots , L_{1n_1}$ to single out an equation
$$\sum\mlimits_{i=1}^m a_iL_{1}G_i+\sum\mlimits_{i=m+1}^{m+\ell} a_i'L_{1}G_i+\cdots=0.$$
In this equality, \multiplefractions $L_{1} G_1, \cdots, L_{1}G_m$ have non-zero coefficients $a_1, \cdots, a_m$, and the weight of each terms in the sum is $N$. Thus by the induction hypothesis we must have $a_i=0, i=1, \cdots m$, which yields the expected contradiction.
\end{proof}

\subsection{The non-holomorphicity of polar germs }
\mlabel{ss:geo2}

Based on Proposition \ref{lem:LinearlyIndep}, we prove the following non-holomorphicity of \gname germs at zero, a central result of this paper.

\begin {thm} $(${\bf Non-holomorphicity Theorem}$)$
A \ppp family
$$\left\{\left.\frac {h_i}{L_{i1}^{s_{i1}}\cdots L_{in_i}^{s_{in_i}}}\right|\, 1\leq i\leq p\right\}$$
of $\coef$-\gname germs at zero is non-holomorphic in the sense that, if a linear combination
\begin{equation}
\sum_{i=1}^p a_i\frac {h_i}{L_{i1}^{s_{i1}}\cdots L_{in_i}^{s_{in_i}}},\quad a_i\in \coef , 1\leq i\leq p,
\mlabel{eq:germ}
\end{equation}
is holomorphic, then $a_i=0$ for $1\leq i\leq p$. In particular, this family of polar germs  is linearly independent and the holomorphic function for the linear combination is identically zero.
\mlabel {thm:non-holo}
\mlabel{thm:LinearlyIndepOfGerms}
\end{thm}

\begin {proof} It suffices to show the theorem for  $\F=\R$ since the result then follows for any subfield  $\F$ of $\R$.

By Lemma~\mref {lem:SepSpace}, we can assume that, for each $1\leq i\leq r$, the weight $\vert s_i\vert =s_{i1}+\cdots +s_{in_i}$ is the same and the linear forms in the denominators span the same space.
As in the proof of Lemma~\mref {lem:SepSpace}, we can pick values  $\ell _{n_1+1}^0, \cdots, \ell _k^0$ of $\ell _{n_1+1}, \cdots, \ell _k$ such that $h_i(\ell _{n_1+1}^0, \cdots, \ell _k^0)\not =0, 1\leq i \leq t$, and $$\sum_{i=1}^t a_i\frac
{h_i(\ell _{n_1+1}^0,\cdots,\ell _k^0)}{L_{i1}^{s_{i1}}\cdots L_{in_i}^{s_{in_i}}}=0.$$
 But the set of fractions $\frac {1}{L_{i1}^{s_{i1}}\cdots L_{in_i}^{s_{in_i}}}, 1\leq i\leq t,$ is \ppp. Thus by Proposition~\mref{lem:LinearlyIndep}, the coefficients $a_ih_i(\e_{n_1+1}^0,\cdots,\e_k^0)$ and hence the coefficients $a_i, 1\leq i\leq t,$ are zero which leads to  a contradiction.
\end{proof}

A direct consequence of Theorem~\mref {thm:non-holo} is the following uniqueness result.

\begin{coro}
Let $\{S_i\}_{1\leq i\leq \ell}$ and $\{T_j\}_{1\leq j\leq m}$ be \ppp families of $\coef$-\gname germs at zero sharing the same properly positioned family of supporting cones (upon a suitable choice of signs of the linear forms). If
\begin{equation}
\sum_{i=1}^\ell S_i+g_0=\sum_{j=1}^m T_j+h_0
\mlabel{eq:uniq}
\end{equation}
for holomorphic germs $g_0$ and $h_0$, then $\{S_i\}_{1\leq i\leq \ell}=\{T_j\}_{1\leq j\leq m}$ and $g_0=h_0$.
\mlabel {coro:uniqueness}
\end{coro}

\begin{proof} Let $\{C_1, \cdots , C_r\}$ be a properly positioned family of supporting cones of $\{S_i\}_{1\leq i\leq \ell}$ and of $\{T_j\}_{1\leq j\leq m}$.
For $1\leq i\leq r,$ let $L_{i1},\cdots, L_{in}$ be fixed generators of the $\F$-cones $C_i$.
Let $N$ be the largest sum of powers in the denominators of $\{S_i\}_{1\leq i\leq \ell}$ and $\{T_j\}_{1\leq j\leq m}$ and denote
$$\{M_1, \cdots, M_t\}=\Big\{L_{i1}^{s_1}\cdots L_{in}^{s_n} \ \Big|\ i=1, \cdots, r, \ \ |\vec s|:=\sum_j s_j \le N \Big\}.$$
Then we have
$$ \sum_{i=1}^\ell S_i = \sum_{k=1}^t \frac{g_k}{M_k}\quad \text{and} \quad \sum_{j=1}^m T_j =\sum_{k=1}^t \frac{h_k}{M_k},
$$
where for $1\leq k\leq t$, $g_k$ and $h_k$, some of which can be zero, are holomorphic in some linear forms orthogonal to the linear forms in $M_k$ with respect to the given inner product. Thus Eq.~(\mref{eq:uniq}) gives
$$ \sum_{k=1}^t \frac{g_k-h_k}{M_k}=h_0-g_0.$$
But the terms in the sum satisfy the conditions of Theorem~\mref{thm:non-holo}. Thus we have $g_k=h_k$ for $0\leq k\leq t$ which implies that the $S_i$'s  match with the $T_j$'s, giving the identification we want.
\end{proof}

\section{Laurent expansions of meromorphic germs at zero with linear poles}
\mlabel {sect:decom}

In this section, we apply the Non-holomorphicity Theorem~\mref{thm:non-holo} and Corollary~\mref{coro:uniqueness} to develop a notion of Laurent expansions for multivariate meromorphic germs at zero with linear poles.

Central to the notion of Laurent expansions is  the forgetful map in Definition~\mref{de:polarspace} which gives  formal expansions of meromorphic germs at zero. Taking local cross sections of this map, we first identify the Laurent subspaces in Proposition~\mref{pp:finj}, formally defined in Definition~\mref{defn:relLaurent}. We then show in Theorem~\mref{thm:DecomF} that these Laurent subspaces cover the whole space $\calm_{\F}(V^{\circledast}\ot \C)$ with the help of the surjectivity of $\varphi$ proved in  Theorem \ref{thm:surj}. We then establish a consistency property of these Laurent subspaces in Proposition~\mref{pp:SubDDecom}. Finally, we determine the kernel of the forgetful map in Section~\mref{sec:kernel}.

\subsection {The  space of formal expansions}
\mlabel{ss:lauexp}

We first generalize the concept of decorated smooth cones in~\cite{GPZ1}.

\begin {defn}
\mlabel {defn:DSimplicialCone}
A {\bf decorated simplicial $\F $-cone} is a formal monomial $\decc:=\cone{v_1}^{s_1}\cdots \cone{v_n}^{s_n}$ where $v_1, \cdots, v_n$ are linearly independent $\F $-vectors and $s_1, \cdots , s_n$ are in $\Z _{\ge 1}$.
The simplicial  $\F$-cone $\cone{v_1,\cdots,v_n}$ generated by $v_1, \cdots, v_n$ is called the {\bf geometric cone} of the decorated cone $\decc$, and is denoted by $G(\decc)$.
\end{defn}

As before, these generators define linear functions $L_1,\cdots, L_n$ on $V_k^*\ot \C$.
For a different choice of the spanning vectors $v_1,\cdots,v_n$, the function $\vec{L}_\decc:=L_1^{s_1}\cdots L_n^{s_n}$ alters by a constant. Thus for any subspace $U$ of $V\ot \C$, the subspace $\frac{1}{{\vec L}_\decc}\calm  _{\F}(U^*)$ does not depend on the choice of the spanning vectors $v_1,\cdots,v_n$. In particular this holds for $U={\rm lin}^\perp(G(\decc))$, the orthogonal complement of the linear span of $G(\decc)$ in $V\ot \C$ with respect to the given inner product $Q$. The space
$$\calm_\decc:=\frac{1}{\vec L_{\decc}}\calm  _{\F,+} \big(({\rm lin}^\perp(G(\decc)))^*\big)\subseteq \calm_\F(V^\circledast\ot \C)$$
is precisely the space spanned by polar germs whose support is $G(\decc)$ and with the fixed denominator $\vec{L}_\decc$.

\begin{defn} \begin{enumerate}
\item
Define the {\bf space of formal expansions in polar germs  } to be
$$\M_{\coef}(V^\circledast \otimes \C):=\Big(\bigoplus_{\decc}\calm_\decc\Big)\,\bigoplus\,{\calm} _{\F, +}(V^\circledast\otimes \C)
=\Big(\bigoplus_{\decc} \frac{1}{\vec L_{\decc}}\calm  _{\F,+} \big(({\rm lin}^\perp(G(\decc)))^*\big) \Big)\,\bigoplus\,{\calm} _{\F, +}(V^\circledast\otimes \C),$$
 where the sum is taken over decorated simplicial $\coef$-cones $\decc$.
\item
Define the {\bf forgetful map}
\begin{equation}
\varphi: \M_{\coef}(V^\circledast\otimes \C)  \longrightarrow {\mathcal M}_{\coef}(V^\circledast\otimes \C), \quad \bigoplus\limits_{\decc} S_\decc \oplus h \mapsto \sum_{\decc} S_\decc + h, S_\decc \in \calm_\decc, h\in {\calm} _{\F, +}(V^\circledast\otimes \C),
\mlabel{eq:varphi}
\end{equation}
sending direct sums in $\M _\F(V^\circledast\otimes \C)$ to sums of  functions in $\calm_\coef(V^\circledast \otimes \C)$
\end{enumerate}
\label{de:polarspace}
\end{defn}

Notice that different decorated simplicial $\F$-cones might give the same space $\calm_\decc$, for example when the generators of a cone change signs, giving multiple copies of identical summand in $\M_{\coef}(V^\circledast \otimes \C)$. For instance, $\calm_{\langle e_1\rangle}=\calm_{\langle - e_1\rangle}$ but they give distinct summands in $\M_{\coef}(V^\circledast \otimes \C)$.

By definition, the restriction of $\varphi$ to $\frac{1}{\vec L_{\decc}}\calm  _{\F,+} (({\rm lin}^\perp(G(\decc)))^*)$ for each decorated simplicial cones $\decc$, as well as to ${\calm} _{\F, +}(V^\circledast\otimes \C)$, is injective. The Non-holomorphicity Theorem~\mref{thm:non-holo} shows that this injectivity of $\varphi$ holds for much larger subspaces of $\M _\F(V^\circledast\otimes \C)$.

\begin{prop}
Let $\conefamilyc$ be a properly positioned family of cones in $V$. Denote
\begin{equation}
\M_{\conefamilyc,-}(V^\circledast \ot \C):= \bigoplus_{G(\decc)\in \conefamilyc} \calm_\decc = \bigoplus_{G(\decc)\in \conefamilyc} \frac{1}{\vec L_{\decc}}\calm  _{\F,+} (({\rm lin}^\perp(G(\decc)))^*).
\mlabel{eq:mc-}
\end{equation}
The restriction of $\varphi$ to
\begin{equation}\label{eq:Mconefamilyk0}
\M_{\conefamilyc}(V^\circledast \ot \C):= \M_{\conefamilyc,-}(V^\circledast \ot \C) \,\oplus\,{\calm} _{\F, +}(V^\circledast\otimes \C) \subseteq \M_\coef (V^\circledast \otimes \C)
\end{equation}
is injective.
\mlabel{pp:finj}
\end{prop}

\begin{proof}
This follows directly from Corollary~\mref{coro:uniqueness}.
\end{proof}

\begin{remark}
\begin{enumerate}
\item As a consequence of the proposition, we have
\begin{equation}
\calm_{\conefamilyc,-}(V^\circledast\ot \C):= \varphi(\M_{\conefamilyc,-}(V^\circledast\ot\C)) \cong \M_{\conefamilyc,-}(V^\circledast\ot \C).
\mlabel{eq:mc--}
\end{equation}
\item Note that $\calm_{\conefamilyc,-}(V^\circledast\ot \C)$ is the space spanned by polar germs   whose supporting cone is contained in $\conefamilyc$.
\end{enumerate}
\end{remark}

\begin{defn}\mlabel{defn:relLaurent} Let ${\conefamilyc}$ be a properly positioned family of simplicial cones. A meromorphic germ $f\in \calm_{\coef}(V_k^*\otimes \C)$ is said to {\bf admit a Laurent expansion supported on ${\conefamilyc}$} if it is contained in $\varphi(\M_{\conefamilyc}(V^\circledast \ot \C))$ or, more precisely, if there exists a \ppp family $\{S_j\}_{j\in J}$ of polar germs  whose supporting cones are contained in $\conefamilyc$, together with a holomorphic germ $h$, all with coefficients in  $\F$, such that
$$f= \varphi\Big(\bigoplus _{j\in J} S_j\oplus h\Big).
$$
The element $\oplus _{j\in J} S_j\oplus h\in\M_{\coef}(V^\circledast \otimes \C)$ with this property, unique by the injectivity in Proposition~\mref{pp:finj}, is called the ${\conefamilyc}$-supported {\bf Laurent expansion} of $f$, denoted by $\mathfrak L_{\conefamilyc}(f)$, that is,
\begin{equation}\mlabel{eq:Laurentf}
\mathfrak L_{\conefamilyc}(f):=\bigoplus _{j\in J} S_j\oplus h.
\end{equation}
The subspace $\varphi(\M_{\conefamilyc}(V^{\circledast}\ot \C))$ of $\calm(V^{\circledast}\ot \C)$ is called the {\bf Laurent subspace} supported by $\conefamilyc$.
\end{defn}

\begin{remark}
\begin{enumerate}
\item Clearly, for a polar germ $f=\frac {h(\ell_1, \cdots , \ell_m)}{L_1^{s_1}\cdots L_n^{s_n}}$ with supporting cone $C$ and a properly positioned family $\conefamilyc$ of simplicial cones containing $C$, we have
$\mathfrak L_{\conefamilyc}(f)=f$.
\item For any $f\in \calm_{\coef}(V_k^*\otimes \C)$ which admits a ${\conefamilyc}$-supported Laurent expansion, we have
$$\varphi\circ \mathfrak L_{\conefamilyc}(f)=f.$$
\end{enumerate}
\end{remark}

\begin{exam}
Take $\conefamilyc=\{\langle e_1\rangle\}$ in the standard Euclidean space. Then polar germs at the variables given by $z=\sum \e _ie_i^*$ and supported on $\conefamilyc$ are $h_i/\e _1^i, i\geq 0$, for holomorphic functions $h_i$ in variables other than $\e _1$. Thus the Laurent subspace supported by $\conefamilyc$, restricted to $V_1:=\R e_1$, is precisely
$$ \calm_\conefamilyc(V_1\ot \C):= \calm_\conefamilyc(V^\circledast\ot \C)\cap \calm(V_1\ot \C) =
 \bigoplus_{i\geq 1} \C\frac{1}{\e^i} \oplus \C\{\{\e\}\},$$
recovering the classical Laurent series expansions.
\mlabel{ex:one}
\end{exam}

\subsection{The subdivision operators} In order to show that every element in $\calm_{\coef}(V^\circledast \otimes \C)$ admits a Laurent expansion, we want to cover $\calm_{\coef}(V^\circledast \otimes \C)$ with Laurent subspaces. This is achieved by the subdivision operators which will also take care of the consistency on overlaps of Laurent subspaces.

The following definition generalizes the concept of a subdivision of a cone.

\begin{defn}\label{defn:properpos} A {\bf subdivision of a family of cones $\{C_i\}$} is a set $\{D_1,\cdots,D_r\}$ of cones such that
\begin{enumerate}
\item
$D_1,\cdots,D_r$ intersect along their faces,
\item
for any $i$, there is $I_i\subset [r]$ such that $\{D_\ell\}_{\ell\in I_i}$ is a subdivision of $C_i$, and
\item
$\cup_{i}I_i=[r]$.
\end{enumerate}
\end{defn}

We introduce a notion which will be convenient for later discussions.

\begin{defn}
Fix an ordered basis $\{e_i\}$ of a filtered space $V=\cup_{k\geq 1} V_k$ such that $\{e_i\}\cap V_k$ is a basis of $V_k$. A nonzero vector $v=\sum_i c_ie_i$ is called {\bf pseudo-positive} if the leading coefficient of $v$, namely the nonzero coefficient of $v$ with the largest subscript $i$, is positive. By convention, $0$ is taken to be a pseudo-positive vector. Let $\mathbf{P}$ denote the set of pseudo-positive vectors.
\mlabel{de:pp}
\end{defn}

As can be easily checked, the set $\mathbf{P}$ is the union of the  increasing filtration consisting of the strictly convex sets $\mathbf{P}_n\subseteq \R \{e_1,\cdots,e_n\}, n\geq 0, $ where, by convention, $\mathbf{P}_0:=\{0\}$ and recursively,
$$ \mathbf{P}_{n+1}:=\mathbf{P}_n\cup (\R\{e_1,\cdots,e_n\}\times \R_{>0} e_{n+1}), \quad n\geq 0.$$
Consequently, $\mathbf{P}$ is a strictly convex set.

\begin{lemma}
Any finite family of cones whose union does not contain a nonzero linear subspace has a properly positioned family of cones as a subdivision. The union of the family of the cones does not change with the subdivision. In particular, if a finite family of cones is in $\mathbf{P}$, then so is any of its properly positioned families of subdivisions.
\mlabel{lem:ppsubd}
\end{lemma}

\begin{proof}
The existence of a subdivision follows the proof of Lemma 2.3(a) in \cite {GPZ1}, noting that the assumption made there, namely that the cones span the same linear subspace, is redundant. Then the assumption that the union of the family does not contain a nonzero linear subspace guarantees that the resulting family is properly positioned. The second other statement also follows from the proof of~\cite[Lemma~2.3(a)]{GPZ1}.
\end{proof}

\begin{lemma}
Given any finite family of polar germs, there is a choice of the family of supporting cones whose union does not contain a nonzero linear subspace.
\mlabel{lem:supconv}
\end{lemma}

\begin{proof}
Fix an ordered basis of $V$. By rescaling if necessary, we can assume that all linear forms in the denominators of the polar germs  are pseudo-positive. The supporting cones of the polar germs   spanned by vectors corresponding to these linear forms are therefore contained in the strictly convex set $\mathbf{P}$ and hence does not contain any non-zero linear subspace.
\end{proof}

\begin{defn} A {\bf pan-subdivision of a family of cones $\conefamilyc=\{C_i\}$} is a set $\conefamilyd=\{D_1,\cdots,D_r\}$ of cones that satisfies conditions~(a) and~(b) in Definition~\ref{defn:properpos}, namely
\begin{enumerate}
\item
$D_1,\cdots,D_r$ intersect along their faces,
\item
for any $i$, there is $I_i\subset [r]$ such that $\{D_\ell\}_{\ell\in I_i}$ is a subdivision of $C_i$.
\end{enumerate}
If all the cones are $\F$-cones, then the pan-subdivision is called a $\F$-pan-subdivision.
\end{defn}
\begin{ex} A subdivision for a family $\conefamilyc$ of cones is a pan-subdivision for a sub-family of $\conefamilyc$.
\end{ex}

Let $\conefamilyc$ be a properly positioned family of cones and $\conefamilyd$ a simplicial pan-subdivision of $\conefamilyc$. We next define a subdivision operator
$$\frakS_{(\conefamilyc,\conefamilyd)}:\M_\conefamilyc (V^\circledast \ot \C)\to \M_\conefamilyd (V^\circledast \ot \C).$$
Since $\M_\conefamilyc (V^\circledast \ot \C):=\oplus_{G(\decc)\in \conefamilyc} \calm_\decc$, we only need to define its action on $\calm_{\decc}$ for a decorated cone $\decc:=\cone{v_1}^{s_1}\cdots \cone{v_n}^{s_n}$ as in Definition~\mref{defn:DSimplicialCone}, with $G(\decc)$ in $\conefamilyc$.

We first consider the action when $s_1=\cdots=s_n=1$. Then a polar germ in $\calm_{\decc}$ is of the form $\frac{g}{L_1\cdots L_n}$ for a simple fraction $\frac{1}{L_1\cdots L_n}$ and a holomorphic germ $g$ in a set of variables orthogonal to the linear span of $L_1,\cdots, L_n$. Let $G(\decc)=C$. There is a unique subset $\{D_\mu\}_{\mu \in J}$ of $\conefamilyd $ that gives a subdivision of $C$. As in Eq.~(\mref{eq:phi0}), we have
$$ I(C)=(-1)^n\frac{a}{L_{1}\cdots L_{n}}, \quad
I(D_\mu)=(-1)^n\frac{b_\mu}{M_{\mu 1}\cdots M_{\mu n}},$$
where $a, b_\mu$ are constants in $\coef$. By Eq.~(\mref{eq:phi1}),
\begin{equation}\label{eq:LM}
\frac{1}{L_{1}\cdots L_{n}}=\frac{(-1)^n}{a}I(C)=\frac{(-1)^n}{a}\sum_{\mu\in J} I(D_\mu)=\sum_{\mu\in J} \frac{b_\mu}{a}\frac{1}{M_{\mu 1}\cdots M_{\mu n}}.
\end{equation}
Note that $I(D_\mu)$ is supported on $D_\mu$, and by Lemma~\mref{lem:gencone2}, such a decomposition with support on $\conefamilyd$ is unique. Thus we can define
\begin{equation}
\frakS_{(\conefamilyc,\conefamilyd)}\Big(\frac{g}{L_1\cdots L_n}\Big):= \bigoplus_{\mu\in J} \frac{b_\mu}{a}\frac{g}{M_{\mu 1}\cdots M_{\mu n}} \in \bigoplus_{\mu\in J} \calm_{D_\mu}\subseteq \M_\conefamilyd (V^\circledast\ot \C).
\mlabel{eq:ssimple}
\end{equation}

We next introduce a class of differential operators in order to treat general decorated cones. Let $\{e_1, e_2, \cdots \}$ be a basis of the filtered space $V$ and let $\{e^*_1, e^*_2, \cdots \}$ be the dual basis. Let $\vec \e=\sum \e _i e_i^*$ be a generic vector in $ V^\circledast\otimes \C$. Then respect to the variables $\e_i$, we have differential operators
$$\partial_i:=-\frac{\partial}{\partial \e _i}: \calm _\coef (V^\circledast\otimes \C) \to \calm_\coef  (V^\circledast\otimes \C).$$
For a fixed vector $v^*=\sum_i c_ie_i^*\in V^\circledast$, denote $\partial _{v^*}:=\sum_i c_i\partial _ i,$ the negative of the directional derivation. Then for any function $f$ in linear independent linear forms $K_1,\cdots, K_m$, the chain rule gives
\begin{equation}
\partial _{v^*}f(K_1,\cdots, K_m)=-\sum_{i=1}^m (v^*, K_m)\frac{\partial f}{\partial K_m}.
\mlabel{eq:diffdefg}
\end{equation}

Now for any given decorated cone $\decc$ with $G(\decc)\in \conefamilyc$ and polar germ $\frac{g}{L_1^{s_1}\cdots L_n^{s_n}}$ in $\calm_\decc$,
let $\{L_i^*=\sum_j
c_{ij} e_j^*\}_i$ be dual to the linear forms $\{L_i\}_i$ in the sense that $(L_i, L_j^*)=\delta
_{ij}, 1\leq i, j\leq n$.
By Eq.~(\mref{eq:diffdefg}) we obtain
$$
\partial_{L_i^*}\frac{1}{M_{1}^{r_1}\cdots M_{n}^{r_n}}
 = \sum_{j=1}^n \frac{c_{ij} }{M_1^{r_1}\cdots M_j^{r_j+1}\cdots M_n^{r_n}},$$
for some constants $c_{i1} ,\cdots, c_{in} $  depending on the poles $M_1, \cdots ,M_n$ and on $L_i$.
Since $g$ is in a set of variables orthogonal to $M_{1},\cdots,M_{n}$,  we further obtain
\begin{equation}
\partial_{L_i^*}\frac{g}{M_{1}^{r_1}\cdots M_{n}^{r_n}}
 = \sum_{j=1}^n \frac{c_{ij} \, g\,}{M_1^{r_1}\cdots M_j^{r_j+1}\cdots M_n^{r_n}} = g \,\partial_{L_i^*}\frac{1}{M_{1}^{r_1}\cdots M_{n}^{r_n}}.
\mlabel{eq:s2}
\notag
\end{equation}
We then define
\begin{equation}
\delta_{L_i^*}: \calm_{\underline{D}} \to \bigoplus_{G(\underline{E})\in \conefamilyd} \calm_{\underline{E}} \subseteq \M_{\conefamilyd}(V^\circledast\ot \C), \quad
 \frac{g}{M_{1}^{r_1}\cdots M_{n}^{r_n}}\mapsto \bigoplus_{j=1}^n \frac{c_{ij}\,g}{M_1^{r_1}\cdots M_j^{r_j+1}\cdots M_n^{r_n}},
\mlabel{eq:s3}
\end{equation}
which, by acting componentwise in $\M_\conefamilyd(V^\circledast \ot \C):=\oplus_{G(\underline{D})\in \conefamilyd} \calm_{\underline{D}}$, gives rise to an operator
\begin{equation}
\delta_{L_i^*}: \M_\conefamilyd(V^\circledast \ot \C) \to \M_\conefamilyd(V^\circledast \ot \C).
\mlabel{eq:delta}
\end{equation}

By \cite[Proposition~4.8 (b)]{GPZ1},
we have
\begin{equation}
\frac{1}{L_1^{s_1}\cdots L_n^{s_n}}
= \frac
1{(s_1-1)!\cdots (s_n -1)!}\partial_{L_1^*}^{s_1-1}\cdots
\partial_{L_n^*}^{s_n-1}\frac{1}{L_1\cdots L_n}.
\mlabel{eq:diffrac}
\end{equation}
We accordingly apply Eqs.~(\mref{eq:ssimple}) and (\mref{eq:delta})   to define
\begin{equation}
\frakS_{(\conefamilyc,\conefamilyd)}\Big(\frac{g}{L_1^{s_1}\cdots L_n^{s_n}}\Big):=
\frac
1{(s_1-1)!\cdots (s_n -1)!}\delta_{L_1^*}^{s_1-1}\cdots
\delta_{L_n^*}^{s_n-1}\Big (\frakS_{(\conefamilyc,\conefamilyd)}\Big(\frac{g}{L_1\cdots L_n}\Big )\Big),
\mlabel{eq:Sdiffrac}
\end{equation}
completing the definition of the {\bf subdivision operator}
\begin{equation}
\frakS_{(\conefamilyc,\conefamilyd)}: \M_{\conefamilyc}(V^\circledast\ot \C) \to \M_{\conefamilyd}(V^\circledast\ot \C).
\mlabel{eq:famsubdiv}
\end{equation}
Notice that by Eq.~ (\mref {eq:s2}),
\begin{equation}
\frakS_{(\conefamilyc,\conefamilyd)} \left ( \frac{g}{L_1^{s_1}\cdots L_n^{s_n}}\right)=
g\frakS_{(\conefamilyc,\conefamilyd)} \left ( \frac{1}{L_1^{s_1}\cdots L_n^{s_n}}\right).
\mlabel{eq:subop}
\end{equation}

In the definition of the subdivision operator, we choose a basis of $V$ and a dual of the linear forms in the polar germs. The following proposition shows that  this operator does not actually depend on such choices.
\begin {prop} \label{pp:subdoop}
\begin{enumerate}
\item
The subdivision operator $\frakS_{(\conefamilyc,\conefamilyd)}$ is compatible with the forgetful map $\varphi$, i.e.,
${\varphi} \circ \mathfrak{S}_{(\conefamilyc, \conefamilyd)} =\varphi.$
\label{it:compatiblesubd}
\item
The subdivision operator $\mathfrak{S}_{(\conefamilyc, \conefamilyd)}$ does not depends on the choice of the basis of $V$.
\label{it:subchoice}
\end{enumerate}
\end{prop}
\begin{proof}
(\mref{it:compatiblesubd}).
By Eqs.~(\ref{eq:LM}) and (\mref{eq:ssimple}), the desired equation holds for polar germs with $s_1=\cdots=s_n=1$. Since $\varphi\circ \delta_{L_j^*} = \partial_{L_j^*} \circ \varphi$ by construction, the desired equation follows from Eqs.~(\ref{eq:diffrac}) and (\ref{eq:Sdiffrac}).
\smallskip

\noindent
(\ref{it:subchoice}). For a polar germ $f$ supported on $\conefamilyc$,  and  for any choice of the basis of $V$, $\mathfrak{S}_{(\conefamilyc, \conefamilyd)}(f)$ is a sum of polar germs supported on $\conefamilyd$, which equals to $f$ as a function by Item~(\mref{it:compatiblesubd}),   so that   by Corollary \mref {coro:uniqueness}, $\mathfrak{S}_{(\conefamilyc, \conefamilyd)}(f)$ is unique.
\end{proof}

Furthermore, for a simplicial  $\F$-pan-subdivision $\conefamilye$ of $\conefamilyd$, by the transitivity of pan-subdivisions, we obtain
\begin{equation}
\mlabel {eq:SubdOp}
\mathfrak{S}_{(\conefamilyd, \conefamilye)}\circ \mathfrak{S}_{(\conefamilyc, \conefamilyd)}=\mathfrak{S}_{(\conefamilyc, \conefamilye)}.
\end{equation}

Now we show that the Laurent subspaces cover the whole space $\calm_{\coef}(V_k^*\otimes \C)$, proving the existence of an Laurent expansion for any meromorphic germ.

\begin {thm}
\mlabel {thm:DecomF} Let $f$ be an element in $\calm_{\coef}(V_k^*\otimes \C)$.
There exists a properly positioned family of simplicial cones $\conefamilyc$ such that $f$ has a Laurent expansion supported on $\conefamilyc$. In other words, there is a \ppp family of polar germs   $\{S_j\}_{j\in J}$ supported on $\conefamilyc$, together with a holomorphic germ $h$, all with coefficients in  $\F$, such that
\begin{equation}\label{eq:decgermf}f= \varphi\left(\bigoplus _{j\in J} S_j\oplus h\right),
\end{equation}
or as function decomposition,
\begin{equation}\label{eq:decgerm}f= \sum _{j\in J} S_j+ h.
\end{equation}
In fact, the family $\conefamilyc$ can be taken to be in $\mathbf{P}$.
\end{thm}

\begin{proof}
Take any decomposition of $f$ as in Theorem~\mref{thm:surj},
$f=\sum _{ i\in I} g_{i} + h,
$
where $\{g_i\}$ is a finite set of polar germs  and $h$ is holomorphic at zero. By Lemma~\mref{lem:supconv}, there is a choice $\conefamilyc$ of the family of the supporting cones of the polar germs such that the union of the cones does not contain any nonzero linear subspace. In fact, the proof of Lemma~\mref{lem:supconv} shows that $\conefamilyc$ can be chosen to be contained in $\mathbf{P}$. By combining colinear terms, we can assume that the decorated cones of these polar germs are distinct. Hence we can write
$f=\varphi(\oplus _{ i\in I} g_{i} \oplus h)$.

By Lemma~\mref{lem:ppsubd}, the family $\conefamilyc$ has a pan-subdivision $\conefamilyd$ that is properly positioned. Then through the subdivision operator $\mathfrak{S}(\conefamilyc ,\conefamilyd)$, the sum
$\bigoplus _{ i\in I} \mathfrak{S}(\conefamilyc ,\conefamilyd) (g_{i}) \oplus h
$
is a desired Laurent expansion of $f$ supported on $\conefamilyd$.
\end{proof}

\begin{exam} In the standard Euclidean space,
we have
$$
\frac {z_1+2z_2}{z_1(z_1+z_2)z_2}=\frac 1{z_1z_2}+\frac 1{z_1(z_1+z_2)} =2\frac {1}{z_1(z_1+z_2)}+\frac 1{(z_1+z_2)z_2}.
$$
Here the first equation expressed the meromorphic germ as a sum of polar germs as in Theorem~\mref{thm:surj}. The second equation rewrite the sum of polar germs as a sum of \ppp family of polar germs, as in Theorem~\mref{thm:DecomF}.
\mlabel{ex:lauexp}
\end{exam}

We finally prove the coherence of Laurent expansions arising from different properly positioned family of cones, namely their compatibility with the subdivision operators.

\begin{prop}
\begin{enumerate}
\item
Assume that $f \in {\mathcal M} _\F(V^\circledast\otimes \C)$ admits a ${\conefamilyc}$-supported Laurent expansion and let ${\conefamilyd}$ be a simplicial  $\F$-pan-subdivision of ${\conefamilyc}$. Then $\mathfrak{S}_{(\conefamilyc , \conefamilyd )}\mathfrak L_{\conefamilyc}(f)$ is the ${\conefamilyd}$-supported Laurent expansion of $f$.
\mlabel{it:trans}
\item
With respect to the inclusion operators, the set of Laurent subspaces supported on cones in the set $\mathbf{P}$ forms a direct system. Its direct limit is $\calm_{\F}(V^\circledast\ot \C)$.
\mlabel{it:direct}
\end{enumerate}
\mlabel{pp:SubDDecom}
\end{prop}

\begin{proof}
(\mref{it:trans}) follows in a straightforward manner from Proposition~\ref{pp:subdoop}.(\ref{it:compatiblesubd}).
\smallskip

\noindent
(\mref{it:direct}) For two properly positioned families of simplicial cones in $\mathbf{P}$, their union is contained in $\mathbf{P}$. Thus by Lemma~\mref{lem:ppsubd}, their union has a properly positioned subdivision of simplicial cones, giving a common pan-subdivision of the two families. Thus the set of properly positioned families of simplicial cones in $\mathbf{P}$ is direct with respect to pan-subdivisions.
Through the subdivision operators, the set $$\{\M_{\conefamilyc}(V^\circledast \ot \C)\,|\, \text{properly positioned families }\conefamilyc \text{ in } \mathbf{P}\}$$
is a direct system.
Then by Proposition~\mref{pp:subdoop}.(\mref{it:compatiblesubd}) the set of Laurent subspaces $$\{\calm_{\conefamilyc}(V^\circledast \ot \C)\,|\, \text{properly positioned families }\conefamilyc \text{ in } \mathbf{P}\}$$
is a direct system with respect to the inclusion maps. Its direct limit is $\calm_{\F}(V^\circledast\ot \C)$ since the union of Laurent subspaces supported in $\mathbf{P}$ is $\calm_{\F}(V^\circledast\ot \C)$ by Theorem~\mref{thm:DecomF}.
\end{proof}

As an immediate consequence, we obtain the following Rota-Baxter type decomposition utilized in~\mcite{GPZ2}.

\begin{cor}\mlabel{cor:DecomF} Let $(V, \Lambda_V)$ be a filtered $\F$-Euclidean lattice space.
There is a direct sum decomposition
$$\calm_\coef (V^\circledast\otimes \C)=\calm  _{\F,-}(V^\circledast\otimes \C)\oplus \calm  _{\F,+} (V^\circledast\otimes \C). $$
In particular, the holomorphic part $h$ and the polar part $\sum_j S_j$ in  Eq.~(\ref{eq:existencedec}) are  uniquely determined by the germ $f$.
\mlabel{it:merodec2}
\end{cor}
\begin{proof}
For each properly positioned family $\conefamilyc$ of simplicial cones, Proposition~\mref{pp:finj} gives the direct sum decomposition
$$ \calm_{\F,\conefamilyc}(V^\circledast \otimes \C)= \calm_{\F,\conefamilyc,-}(V^\circledast\otimes \C) \oplus \calm_{\F,+}(V^\circledast\otimes \C).$$
By Proposition~\mref{pp:SubDDecom}, we obtain
$$\calm_{\F}(V^\circledast\ot \C) = \dirlim \calm_{\F,\conefamilyc}(V^\circledast\ot \C) =
\dirlim\calm_{\F,\conefamilyc,-}(V^\circledast\ot \C) \oplus \calm_{\F,+}(V^\circledast\ot \C)
=\calm_{\F,-}(V^\circledast\ot \C) \oplus \calm_{\F,+}(V^\circledast\ot \C),$$
where the direct limits are taken over those $\conefamilyc$ in $\mathbf{P}$.
\end{proof}

We introduce a notation before stating the next result.
\begin{defn}
Germs $f, g\in \calm_\coef(V^\circledast\ot \C)$ are said to be {\bf orthogonally variate germs} if there are germs $\tilde f $ on $\C ^n$ and $\tilde g$ on $\C ^m$ such that $ f=\tilde f(L_1,\cdots,L_m)$ and $g=\tilde g(M_1,\cdots,M_n)$ for linear independent linear forms $\{L_1,\cdots,L_m\}$ and $\{M_1,\cdots,M_n\}$ on $V^\circledast\ot \C$ with $Q(L_i,  M_j)=0$ for $(i,j)\in [1,m]\times [1,n]$.
\mlabel{de:orth}
\end{defn}

\begin{cor} \mlabel{cor:proj}
{\bf (Multiplicativity of $\pi_+$ on orthogonally variate germs)}
Let
\begin{equation}\pi_+ : \calm  _\F (V^\circledast\otimes \C)\to \calm  _{\F,+}(V^\circledast\otimes \C)
\mlabel{eq:merodec}
\end{equation}
denote the projection map onto $\calm  _{\F,+}(V^\circledast\otimes \C)$ along $\calm  _{\F,-}(V^\circledast\otimes \C)$. For orthogonally variate germs $f$ and $g$, we have
\begin{equation}
\pi_+(f\, g)= \pi_+(f)\, \pi_+(g).
\mlabel{eq:pimult}
\end{equation}
\end{cor}
\begin{proof}
Let $f$ and $g$ be in $\calm  _{\F,+}(V^\circledast\otimes \C)$. Using  Eq.~(\ref{eq:existencedec}), we decompose
	$f=h+ \sum_{i=1}^mS_i$  and 	$g=k+\sum_{j=1}^nT_j$ with   $h, k$   holomorphic germs and $  S_i$, $T_j$ polar germs. Further by Theorem~\mref{thm:surj}, with the notations in Definition~\mref{de:orth}, $h$ and $S_i$ (resp. $g$ and $T_j$) can be written as functions in linear forms in $\text{span}(L_1,\cdots,L_m)$ (resp. $\text{span}(M_1,\cdots,M_n)$). Now
$$fg=hk+h\Big(\sum_{j=1}^nT_j\Big)+k\Big(\sum_{i=1}^mS_i\Big)+\sum_{i,j}S_i T_j.
$$
By the orthogonality of $\text{span}(L_1,\cdots,L_m)$ and $\text{span}(M_1,\cdots,M_n)$, the germs $hT_j$, $kS_i$ and $S_iT_j$ are all polar germs.
Thus this is a decomposition of $fg$ into the sum of a holomorphic germ $hk$ and a linear combination of polar germs. Thus by Corollary~\mref{it:merodec2}, $\pi_+(fg)=hk=\pi_+(f)\pi_+(g)$.
\end{proof}

\begin{remark}
The projection $\pi_+$ is a multivariate generalization of the minimal subtraction operator in one variable. The multiplicativity on orthogonally variate germs stated in Corollary~\mref{cor:proj} is closely related to locality in quantum field theory and central in renormalization issues.
\end{remark}
\subsection{The kernel of the forgetful map}
\mlabel{sec:kernel}
We finally determine the kernel of the forgetful map
$\varphi: \M_{\coef}(V^\circledast\otimes \C)  \longrightarrow {\mathcal M}_{\coef}(V^\circledast\otimes \C)$
introduced in Eq.~(\mref{eq:varphi}).

\begin{theorem} \label{thm:kernphi} The kernel of the map $\varphi $ is the subspace of $ \M_{\coef}(V^\circledast\otimes \C)$ spanned by elements of the following forms	
\begin {itemize}
\item [$\mathrm{I}$.] $\frac {h(\ell_1, \cdots , \ell_m)}{L_1^{s_1}\cdots L_n^{s_n}} \oplus\left((-1)^{s_1+1} \frac {h(\ell_1, \cdots , \ell_m)}{(-L_1)^{s_1}\cdots L_n^{s_n}}\right)$,	 for all polar germs  of the form $\frac {h(\ell_1, \cdots , \ell_m)}{L_1^{s_1}\cdots L_n^{s_n}}$;
\item [$\mathrm{II}$.] $\frac {h(\ell_1, \cdots , \ell_m)}{L_1^{s_1}\cdots L_n^{s_n}}\oplus \mathfrak{S}_{(\conefamilyc , \conefamilyd )}\left(-\frac {h(\ell_1, \cdots , \ell_m)}{L_1^{s_1}\cdots L_n^{s_n}}\right)$, for all polar germs of the form $\frac {h(\ell_1, \cdots , \ell_m)}{L_1^{s_1}\cdots L_n^{s_n}}$, $\conefamilyc:=\{\langle L_1,\cdots, L_n\rangle\}$ and $\conefamilyd$ a simplicial subdivision of $\langle L_1,\cdots, L_n\rangle$.
\end{itemize}
\end{theorem}

Thus modulo changing of signs,  relations among polar germs   amount to subdivision relations.

\begin {proof} Clearly, the subspace  $W$  of $\M_{\coef}(V^\circledast\otimes \C)$   generated by elements of the forms I and II is a subspace of $ \ker \varphi$. So we only need to prove that if $G\oplus H$  is in $\ker \varphi$ with $G=\oplus_j S_j$ a sum of polar germs   $S_j$ and $H$ a holomorphic germ at zero as in Theorem~\mref{thm:surj}, then $G$ lies in $ W$ and $H$ vanishes.

By Lemma~\mref{lem:supconv}, modulo elements of form I, we can assume that the union of the supporting cones of $S_j$ does not contain any non-zero subspace. Let $\conefamilyc:=\{C_j\,|\, j\in J\}$ be the family of supporting cones, and let $\conefamilyd$ be a simplicial subdivision of $\conefamilyc$. Then $G+ \mathfrak{S}_{(\conefamilyc , \conefamilyd )}(-G)=\sum_j \left(S_j+ \mathfrak{S}_{(\conefamilyc _j , \conefamilyd _j )}(-S_j)\right) $ -- where $\conefamilyc _j$ is the singleton $\{ C_j\}$  and $ \conefamilyd _j$ is the subdivision of $C_j$ induced by $\conefamilyd$ -- is a sum of elements of type II  and hence lies in $ W$. Since $\mathfrak{S}_{(\conefamilyc , \conefamilyd )}(-G)-H=-G-H\in \ker \varphi$, we have
$$\varphi(\mathfrak{S}_{(\conefamilyc , \conefamilyd )}(-G))-\varphi(H)=0.$$
Theorem \mref {thm:non-holo} and Proposition~\mref{pp:finj} then yield $\mathfrak{S}_{(\conefamilyc , \conefamilyd )}(-G)=0$ and $H=0$. Therefore,
$$ G+H=G+\frakS_{\conefamilyc,\conefamilyd}(-G) - \frakS_{\conefamilyc,\conefamilyd}(-G) +H$$
is in $W$.
\end{proof}

\section{Refined gradings and applications}
Laurent expansions have many  useful applications, such as   providing much finer decompositions of $\calm  _\F (V^\circledast\otimes \C)$ than the one in Corollary~\mref {cor:DecomF}. As applications, we obtain the Brion-Vergne decomposition and the Jeffery-Kirwan residue of a class of meromorphic germs.

\subsection{Decompositions of meromorphic germs at zero}
Let $(V,\Lambda_V)$ be a filtered lattice space.
\begin{defn}\begin{enumerate}
\item For a polar germ $\frac {h(\ell_1, \cdots , \ell_m)}{L_1^{s_1}\cdots L_n^{s_n}}$, we call $s_1+\cdots +s_n$ the {\bf p-order} of the polar germ.
\item We call the {\bf supporting subspace} of a polar germ   the subspace spanned by the supporting cone of the polar germ.
\item For $p\in \Z_{\geq 0}$, let  $\calm_{\coef}^p(V^\circledast \otimes \C)$ denote the linear span of $\coef$-polar germs with p-order $p$.
\item For any $\coef $-subspace $U\subset V$ , let  $\calm_{\coef,U}(V^\circledast \otimes \C)$ denote the linear span of $\coef$-polar germs with supporting subspace $U$.
\item For $d\in \Z_{\geq 0}$, let  $\calm_{\coef,d}(V^\circledast \otimes \C)$ denote the linear span of $\coef$-polar germs whose supporting subspaces have dimension $d$.
\item For any $\coef $-subspace $U\subset V$ and $p\in \Z_{\ge 0}$, let  $\calm_{\coef,U}^p(V^\circledast \otimes \C)$ denote the linear span of $\coef$-polar germs with supporting subspace $U$ and p-order $p$.
\end{enumerate}
\mlabel{de:supp}
\end{defn}

\begin{remark} With these notations, we have $\calm_{\coef,\{0\}}(V^\circledast \otimes \C)= \calm_{\coef,0}(V^\circledast  \otimes \C)= \calm_{\coef,+}(V^\circledast \otimes \C)$ for the trivial cone $\{0\}$ and integer $d=0$.
\end{remark}

\begin{theorem}\label{thm:SpaceDec} We have the decompositions
\begin{eqnarray}
{\calm} _{\F }(V^\circledast\otimes \C)&=&\bigoplus_{p\geq 0 } {\calm} _{\F}^p(V^\circledast\otimes \C),\mlabel{eq:pgrad}\\
{\calm} _{\F }(V^\circledast\otimes \C)&=&\bigoplus_{U \subset V } {\calm} _{\F, U}(V^\circledast\otimes \C),\mlabel{eq:ugrad}\\
{\calm} _{\F }(V^\circledast\otimes \C)&=&\bigoplus_{d\geq 0} {\calm} _{\F, d}(V^\circledast\otimes \C),\mlabel{eq:dgrad}\\
{\calm} _{\F }(V^\circledast\otimes \C)&=&\bigoplus_{U \subset V, p\in \Z_{\geq 0}} {\calm} _{\F, U}^p(V^\circledast\otimes \C).
\end{eqnarray}

\end{theorem}
Eq.~(\mref{eq:ugrad}) yields the decomposition in~\cite[Theorem~7.3]{BV2} corresponding to a sum running over the set of subspaces spanned by elements of the hyperplane of arrangements corresponding to the poles.

\begin{proof}
By Theorem~\mref{thm:kernphi}, the kernel of the surjective linear map $\varphi: \M_{\coef}(V^\circledast\otimes \C)  \longrightarrow {\mathcal M}_{\coef}(V^\circledast\otimes \C)$ is linearly spanned by elements each of which is a linear combination of polar germs with the same p-order, the same supporting subspace, the same dimension of the supporting subspace. Then the equations follow.
\end{proof}

On the grounds of Theorem~\ref {thm:SpaceDec}, we can give the following definitions.

\begin {defn}
\mlabel {de:ProjU}
Let $U$ be an $\coef$-subspace of $(V,\Lambda_V)$ and $p\in \Z_{\geq 0}$. Define
$$ P_{U}^p: \calm  _\F (V^\circledast\otimes \C)\to \calm  _{\F, U}^p (V^\circledast\otimes \C)\subset\calm  _\F (V^\circledast\otimes \C) $$
and $$ P_{U}: \calm  _\F (V^\circledast\otimes \C)\to \calm  _{\F, U} (V^\circledast\otimes \C)\subset\calm  _\F (V^\circledast\otimes \C) $$
to be the projections,
called the {\bf projection of $f$  onto the space $U$ of p-order $p$} and {\bf projection of $f$  onto the space $U$} respectively.
\end{defn}

For $d\in \Z_{\geq 0}$, setting
$${\calm} _{\coef ,\leq d }(V^\circledast\otimes \C):= \bigoplus_{0\leq k\leq d}  {\calm} _{\coef, k}(V^\circledast\otimes \C),\quad {\calm} _{\coef, > d }(V^\circledast\otimes \C):= \bigoplus_{k> d}  {\calm} _{\coef, k}(V^\circledast\otimes \C),$$
then we have
\begin {equation}
\mlabel {eqn:DecDim}
{\calm} _{\coef  }(V^\circledast\otimes \C)= {\calm} _{\coef,\leq d }(V^\circledast\otimes \C)\oplus {\calm} _{\coef,> d}(V^\circledast\otimes \C).
\end{equation}
This yields back the decomposition of Corollary \mref{cor:DecomF} if we take $d=0$.

The decomposition in Eq.~(\mref {eqn:DecDim}) also yields back Brion-Vergne's decomposition~\cite[Theorem 1]{BrV} as follows.
Let $\mathbf{\Delta} $   be a finite subset of lattice vectors in some $V$ with coefficients in $\F$. Let
\begin{equation}
U:=\mathrm{span} (\mathbf{\Delta}), \quad r:=\dim (U).
\mlabel{eq:ur}
\end{equation}
The symmetric algebra $S(U)$  {(over $\C$)} can be viewed as the algebra of polynomial functions on $U^*$. Following the notation of \cite{BrV}, let us denote by
$$R_{\mathbf{\Delta} } :=\mathbf{\Delta}^{-1}S(U)
$$
the localization of $S(U)$ with respect to $\mathbf{\Delta}$ which is naturally regarded as a subset of $S(U)$. It corresponds to the algebra of rational functions with linear poles in $\mathbf{\Delta}$. A subset $\kappa \subset \mathbf{\Delta}$ is called {\bf generating} if the linear span of $\kappa $ is $U$, and it is  called {\bf a basis} if it is a basis of $U$. Consider the following subspaces of $R_{\mathbf{\Delta}}$:
\begin{align*}
S_{ \mathbf{\Delta}}:=& \mathrm{span}\Big \{  \frac 1{\Pi _{\alpha \in \kappa}\alpha }\,\Big |\, \kappa \subseteq \mathbf{\Delta} \text{ bases of } U \Big\},\\
G_{ \mathbf{\Delta}}:=& \mathrm{span} \Big \{  \frac 1{\Pi _{\alpha \in \kappa}\alpha ^{n_\alpha} }\,\Big | \,\kappa\subseteq  \mathbf{\Delta} \text{ generating subsets of } U, n_\alpha\in \ZZ_{>0} \Big\},\\
NG_{ \mathbf{\Delta}}:=& \mathrm{span} \Big \{ \frac h{\Pi _{\alpha \in \kappa}\alpha ^{n_\alpha}}\,\Big |\, \kappa \subseteq \mathbf{\Delta} \text{ non-generating subsets of } U, n_\alpha\in \ZZ_{\geq 0}, h \in S(U) \Big\}.
\end{align*}

Clearly,
$$ {\calm} _{\coef, >r-1  } (V^\circledast \ot \C)\cap R_{ \mathbf{\Delta}} =G_{ \mathbf{\Delta}};\quad {\calm}_{\coef, \leq r-1}(V^\circledast \ot \C)\cap R_{ \mathbf{\Delta}} =NG_{ \mathbf{\Delta}}. $$
Thus Eq.~(\mref{eqn:DecDim}) recovers the following decomposition of $R_{\mathbf{\Delta}}$ obtained by Brion-Vergne.

\begin{cor} \cite[Theorem 1]{BrV} \mlabel{thm:BrionVergne} There is a direct sum decomposition
$$R_{ \mathbf{\Delta}}=G_{ \mathbf{\Delta}}\oplus NG_{ \mathbf{\Delta}}.
$$
\end{cor}

\subsection{The generalized  Jeffrey-Kirwan residue}
The Jeffrey-Kirwan residue  introduced in \cite {JK1} (see also \cite{JK2}) in the study of localization for nonabelian compact group actions,  is a powerful tool to compute intersection
numbers for symplectic quotients.

There are several ways to define
the Jeffrey-Kirwan residue, namely using iterated residues, inverse Laplace transforms or nested sets \cite {BrV,dCP1,JK1, JK2,JM,SV}. We will use Brion-Vergne's presentation~\cite{BrV}, which we briefly recall here.

Taking total degrees gives a grading on the space $R_{ \mathbf{\Delta}}=\oplus _{j\in \Z} R_{\mathbf{\Delta}}[j]$
 and $G_{ \mathbf{\Delta}}$ is contained in  $R_{\mathbf{\Delta}}[\le -r]:=\oplus_{j\le -r} R_{\mathbf{\Delta}}[j]$. Thus from Corollary~\mref{thm:BrionVergne} we obtain
\begin{equation}
R_{\mathbf{\Delta}}[\leq -r] = G_{\mathbf{\Delta}} \oplus (NG_{\mathbf{\Delta}}\cap R_{\mathbf{\Delta}}[\leq -r])
\mlabel{eq:rdec}
\end{equation}
Furthermore  $S_{\mathbf{\Delta}}=G_{\mathbf{\Delta}}[-r]$ is the highest degree part of $G_{\mathbf{\Delta}}$, giving the decomposition
\begin{equation}
G_{\mathbf{\Delta}} = G_{\mathbf{\Delta}}[<-r] \oplus S_{\mathbf{\Delta}}.
\mlabel{eq:gdec}
\end{equation}

Consider  the localization
$$\hat {R}_{ \mathbf{\Delta}}:= \mathbf{\Delta}^{-1}\hat {S}(U )$$
  of the ring $\hat {S}(U )$ of formal power series by inverting
the linear functions $\alpha \in \mathbf{\Delta}$ and  the natural decomposition
\begin{equation}
\hat {R}_{ \mathbf{\Delta}} =\hat{R}_{\mathbf{\Delta}}[>-r]\oplus R_{\mathbf{\Delta}}[\le -r].
\mlabel{eq:hatdec}
\end{equation}
Putting Eqs.~(\mref{eq:rdec}) -- (\mref{eq:hatdec}) together yields the decomposition
\begin{equation}
\hat {R}_{ \mathbf{\Delta}} =\hat{R}_{\mathbf{\Delta}}[>-r]\oplus
(NG_{\mathbf{\Delta}}\cap R_{\mathbf{\Delta}}[\leq -r]) \oplus G_{\mathbf{\Delta}}[<-r] \oplus S_{\mathbf{\Delta}}.
\mlabel{eq:alldec}
\end{equation}

\begin {defn}The {\bf Jeffrey-Kirwan residue map}
$${\rm Res}_{ \mathbf{\Delta}}: \hat {R}_{ \mathbf{\Delta}}\to S _{ \mathbf{\Delta}}$$
is defined to be the projection to the direct summand $S_{\mathbf{\Delta}}$ in Eq.~(\mref{eq:alldec}).
\end{defn}

Since the Jeffrey-Kirwan residue of a  Laurent power series is defined by that of the corresponding truncated Laurent polynomial, for the sake of simplicity, we focus here on $R_{\mathbf{\Delta}}$ which is a subspace of $\calm _{\F}(U^*\otimes \C)$, and the decomposition
\begin{equation}
{R}_{ \mathbf{\Delta}} ={R}_{\mathbf{\Delta}}[>-r]\oplus
(NG_{\mathbf{\Delta}}\cap R_{\mathbf{\Delta}}[\leq -r]) \oplus G_{\mathbf{\Delta}}[<-r] \oplus S_{\mathbf{\Delta}}.
\mlabel{eq:alldecr}
\end{equation}
 analogous to Eq.~(\mref{eq:alldec}).

\begin {coro} Let $U=\mathrm{span}(\mathbf{\Delta})$ and $r=\dim U$. Then for any $f\in R_{\mathbf{\Delta}}$,  the projection $P_{U}^r(f)$ from Definition~\mref{de:ProjU} is the Jeffrey-Kirwan residue of $f$.
\mlabel{co:jkr}
\end{coro}

\begin{proof} Let $\frac{1}{\prod_{\alpha\in \kappa} \alpha}$ be a spanning fraction of $S_{\mathbf{\Delta}}$. Then $\kappa$ is a basis of $U$ and $\prod_{\alpha\in \kappa} \alpha$ has degree $r$. Thus the fraction is in $\calm_U^r(V^\circledast\ot \C)$ and hence is fixed by $P_U^r$. On the other hand, the supporting cone for a polar germ in ${R}_{\mathbf{\Delta}}[>-r]$ or $(NG_{\mathbf{\Delta}}\cap R_{\mathbf{\Delta}}[\leq -r])$ does not span $U$, while the polar germs in $G_{\mathbf{\Delta}}[<-r]$ do not have p-order $r$. Hence the polar germs are annihilated by $P_U^r$.
\end{proof}

Motivated by this fact, we set the following definition.
\begin{defn} For a meromorphic germ $f$, and an $\F $-subspace $U$ of $V$, let $d={\rm dim }U$, then $P_{U}^d(f)$ is called the {\bf generalized Jeffrey-Kirwan residue} of $f$ supported on $U$.
\mlabel{de:jkr}
\end{defn}

\section {A filtered residue and a coproduct}
\mlabel{sec:jk}
In this part, we give two  further applications of our Laurent theory developed in Section~\mref{sect:decom}. We study the p-order of a meromorphic germ at zero, and defined an invariant, called p-residue for the germ. We show that for exponential sums, taking p-residue amounts to the exponential integrals.
We also define a coproduct on the space of meromorphic germs at zero with linear poles.

\subsection {The p-order and p-residue}
\mlabel{ss:pred}

The grading in Eq.~(\mref{eq:pgrad}) by p-orders of polar germs   gives a p-order for any elements in $\calm(V^\circledast\ot \C)$.

\begin {defn}
\mlabel {defn:Poder}
Let $f\in \calm (V ^\circledast\otimes \C)$. Let
\begin{equation}
\mathfrak L_{\conefamilyc}(f)=\oplus _{j\in J} S_j\oplus h
\mlabel{eq:fdec}
\end{equation}
be a $\conefamilyc$- supported Laurent expansion of $f$ for some appropriate family of supporting cones $\conefamilyc$ as in Definition~\mref{defn:relLaurent}.
\begin{enumerate}
\item
Define the {\bf polar order}, or {\bf p-order} in short, of $f$ to be
$${\rm \pord}(f):=\max_j({\rm \pord} (S_j)),$$
where $\pord(S_j)$ is from Definition~\mref{de:supp}.
\item
Let $S_i=\frac {h_i}{L_{i1}^{s_{i1}}\cdots L_{in_i}^{s_{in_i}}}, 1\leq i\leq t,$ be the \gname germs in Eq.~(\mref{eq:fdec})
We define the {\bf highest polar order residue}, or the {\bf p-residue} in short, of $f$ to be
$${\rm \pres} (f) {:}=\sum_{i=1}^t \frac {h_i(0)}{L_{i1}^{s_{i1}}\cdots L_{in_i}^{s_{in_i}}}.
$$
\end{enumerate}
\mlabel{defn:pres}
\end{defn}

These notions are well-defined thanks to the following property.

\begin {prop}
The p-order and p-residue of a meromorphic germ with linear poles depend neither on the choice of a Laurent expansion nor on the choice of the inner product used in the decomposition of $\calm (V^\circledast\otimes \C)$ in Theorem~\mref{thm:DecomF}.

 Furthermore, for orthogonally variate $f$ and $g$ in the sense of Definition~ \ref{de:orth}, we have
		$$\pres(f g)=\pres(f)\, \pres(g).$$
\label {pp:PordDec}
\end{prop}
Before giving the proof, let us recall the following elementary yet useful result.
\begin{lemma}
Let $\cali$ be a direct system and let $\varphi$ be a function on $\cali$. If $\varphi(i)=\varphi(j)$ for all $i\leq j$ in $\cali$, then $\varphi$ is a constant.
\mlabel{lem:phi}
\end{lemma}

\begin{proof} (of Proposition ~\mref{pp:PordDec})
The independence of the p-order on the choice of a Laurent expansion follows from the grading in Eq.~(\mref{eq:pgrad}).
From Eq.~(\mref{eq:subop}), the numerator of a polar germ and hence the p-residue of $f$ does not change under the subdivision map $\frakS_{(\conefamilyc,\conefamilyd)}$. Then the independence of the p-residue on the choice of a Laurent expansion follows from Lemma~\mref{lem:phi}.

We next prove the independence of the p-order on the inner product.
For an inner product $Q$ in $V$ and $f\in \calm (V ^\circledast\otimes \C)$ with ${\rm \pord }(f)=p$.
Following Eq.~(\mref{eq:fdec}), we write the Laurent expansion of $f$ supported by $\conefamilyc$ as
\begin{equation}
{\mathfrak L}_{\conefamilyc }(f)=\sum _{i=1}^r S_i+\sum_{j=r+1}^{n} S_j+h,
\mlabel{eq:lpg}
\end{equation}
with the \gname germs sharing the largest p-order $p$ grouped in the first sum and those with lesser p-order in the second sum.

Relative to a different inner product $R$ on $V$, an $S_i$ might not be  a \gname germ any longer.
Set
$$S_i=\frac {h_i(\ell _{i1}, \cdots, \ell _{im_i})}{L_{i1}^{s_{i1}}\cdots L_{in_i}^{s_{in_i}}}
$$
with $Q(\ell _{ip}, L_{iq})=0$. For $j=1, \cdots, m_i$, there are coefficients $a_{ij}^k$ such that
$$\ell _{ij}=\ell'_{ij}-\sum _{k=1}^{n_i}a_{ij}^kL_{ik},
$$
where $R(\ell '_{ij}, L_{ik})=0$ for $k=1, \cdots, n_i$. Then
$$
S_i=\frac {h_i(\ell' _{i1}, \cdots, \ell' _{im_i})}{L_{i1}^{s_{i1}}\cdots L_{in_i}^{s_{in_i}}}+{\rm terms \ of \ lower \ denominator \ degrees}.
$$
Thus
\begin {equation}
f=\sum_{i=1}^r \frac {h_i(\ell' _{i1}, \cdots, \ell' _{im_i})}{L_{i1}^{s_{i1}}\cdots L_{in_i}^{s_{in_i}}}+{\rm terms \ of \ lower \ denominator \ degrees}.
\mlabel {equ:DecomDiffIP}
\end{equation}
This gives a decomposition of $f$ as a linear combination of \gname germs for the inner product $R$.

The supporting cones from the right hand side of the above equation are easily seen to be  faces of the supporting cones in the decomposition of $f$ under the inner product $Q$  arising in Eq.~(\mref{eq:lpg}). So they  remain properly positioned. Since $h_i(\ell _{i1}, \cdots, \ell _{im_i})\not =0$, we also have $h_i(\ell '_{i1}, \cdots, \ell ' _{im_i})\not =0$. Therefore under the inner product $R$, the p-order of $f$ is again $p$.

Furthermore, Eq.~(\mref {equ:DecomDiffIP}) shows  how the \gname germs   of p-order ${\rm \pord}(f)$ change for a different inner product. In particular the constant terms of the numerators remain the same. Thus the p-residue does not depend on the choice of inner products.

The second statement follows from the fact that  the product of a polar germ with either a polar germ or a holomorphic germ which are orthogonally variate is again a polar germ. Thus the highest polar order part of $fg$ is the the product of the highest polar order part of $f$ and that of $g$.
\end{proof}

\begin{remark}
The proof of this proposition actually shows  how the terms of   p-order $p$ change as   the inner products change.
\end{remark}

To simplify the notation, for a \gname germ
$S=\frac {h(\ell _1, \cdots, \ell _m)}{L_1^{s_1}\cdots L_k^{s_k}},$
we set
$$S(0_{\vec{\ell}}):=\frac {h(0_{\vec{\ell}})}{L_1^{s_1}\cdots L_k^{s_k}}.$$

\begin {prop}
Let $f=\sum_i S_i+\sum_j T_j+h$, with $S_i$, $T_j$ \gname germs at zero, $h$ a holomorphic germ at zero, ${\rm \pord}(S_i)$'s all equal to $r$,  $\sum_i S_i \not =0$ and ${\rm \pord}(T_j)<k$. Then ${\rm \pord }(f)=r$ and
$\pres(f)=\sum_i S_i(0_{\vec{\ell}}).$
\mlabel {pp:GenPRes}
\end{prop}
\begin {proof} Taking a subdivision of the set of supporting cones of the germs $S_i$'s and $T_j$'s, we have
$S_i=\sum_{i\ell} S_{i\ell}$ and $T_j=\sum_{jm} T_{jm}$. Then
$f=\sum_{i\ell} S_{i\ell}+\sum_{jm} T_{jm}+h.
$
Combining terms that are proportional to one another, we can assume that this decomposition satisfies the conditions in Theorem \mref {thm:DecomF}. In this decomposition there are no terms of p-order greater than $r$ and the sum of all the terms of p-order $r$ is
$\sum _{i,\ell}S_{i\ell}=\sum _i S_i\not =0.$
Thus  ${\rm \pord }(f)=k$ and
${\rm \pres} (f)=\sum_{i\ell} S_{i\ell}(0_{\vec{\ell}})=\sum_i S_i(0_{\vec{\ell}}).$
\end{proof}

\subsection {The p-residue of the exponential sum on a lattice cone}
\mlabel {sect:ResExpSum}
As in \cite{GPZ2}, we can reinterpret  the constructions of \cite{BV1,GP,La} in terms of lattice cones.

We recall from  \cite{GPZ2} that a {\bf lattice cone }  in $V_k$ is a pair  $(C, \Lambda _C)$  with $C$  a cone in $V_k$ and $\Lambda _C$  a lattice in ${\rm lin}(C)$  generated by lattice vectors.
A lattice cone  $(C, \Lambda _C)$  is called {\bf strongly convex} (resp. {\bf simplicial}) if $C$ is. A lattice cone $(C, \Lambda _C)$ is called {\bf smooth} if the additive monoid $\Lambda_C\cap C$ has a monoid basis. In other words, there are linearly independent lattice vectors $v_1,\cdots,v_\ell$ such that
$\Lambda_C\cap C=\ZZ_{\geq 0}\{v_1,\cdots,v_\ell\}$.

To  a  lattice cone $(C,\Lambda_C)$ we can assign two meromorphic functions. One is the exponential sum
$S(C,\Lambda_C)$~\cite{Ba} (corresponding to  $S^c(C,\Lambda_C)$ in \cite{GPZ2}), given in the strongly convex case by
$$S\lC(\vec \e ): =
\sum_{\vec{n}\in C\cap \Lambda_C} e^{\langle \vec n, \vec \e \rangle}.$$
The other function is the exponential integral $I(C,\Lambda_C)$ \cite {GPZ2}, which is a generalization of Eq.~(\mref{eq:phi0}), where the matrix $A_C$ is with respect to a basis of $\Lambda _C$.

\begin {lemma} For a smooth lattice cone $(C, \Lambda _C)$, we have
$${\rm \pord}(S(C, \Lambda _C))={\rm \pord}(I(C, \Lambda _C))={\rm dim}(C), \quad
{\rm \pres}(S(C, \Lambda _C))=I(C, \Lambda_C).
$$
In fact, we have
$$S(C, \Lambda_{C})=I (C, \Lambda_C) +({\rm terms \ of \ p\text{-}order < \ dim}(C)).
$$
\mlabel{lem:por}
\end{lemma}
\begin {proof} Let $v_1, \cdots, v_d$ (where $d=\dim C$) be a basis of $\Lambda _C$ that generates $C$ as a cone. Then
$$S(C, \Lambda_{C})(\vec \e)=\prod_{i=1}^d \frac 1{1-e^{\langle v_i, \vec \e\rangle }}=\prod_{i=1}^d \Big(-\frac 1{\langle v_i, \vec \e\rangle }+h(\langle v_i, \vec \e\rangle )\Big),
$$
where $h$ is holomorphic. So the highest p-order term is
$\prod_{i=1}^d \left(-\frac 1{\langle v_i, \vec \e\rangle } \right)$
which is $I (C, \Lambda_C)$ and has p-order $d$.
\end{proof}

\begin {lemma} For a lattice cone $(C,\Lambda_C)$,
$$I (C,\Lambda_C)\not =0\Longleftrightarrow S\lC \not= 0\Longleftrightarrow C \,\text{ is strongly convex}.$$
\end{lemma}

\begin {proof} We already know that $I\lC=0$ and $S\lC=0$ if $C$ is not strongly convex {\cite {GPZ2}}.
So we only need to prove that if $C$ is strongly convex, then $I\lC \not =0$ and $S\lC\neq 0$.

Take a smooth subdivision $\{C_i\}$ of $C$. Since $C$ is strongly convex, $\{C_i\}$ is properly positioned. So the $I(C_i, \Lambda _{C_i})$'s are linearly independent by the Non-holomorphicity Theorem~\mref{thm:non-holo}. Then $I(C,\Lambda_C)$, as their sum, can not be $0$.

Further, note that
\begin{align*}
S(C, \Lambda _C)&=\sum_i S(C_i, \Lambda_{C})+({\rm terms \ with \ p}\mbox{-}\text{order} < \dim(C))\\
&=\sum_i S(C_i, \Lambda_{C})=I (C_i, \Lambda_{C}) +({\rm terms \ with \ p\mbox{-}order < \ dim}(C)).
\end{align*}
By Proposition~\mref {pp:GenPRes}, ${\rm \pord}(S(C, \Lambda _C))<{\rm dim(C)}$ implies
$\sum_i I (C_i, \Lambda _C)=0,
$
which is a contradiction. Then ${\rm \pord}(S(C, \Lambda _C))={\rm dim(C)}$ and so $S \lC\not =0$.
\end{proof}

\begin{prop}\label{thm:ressubd} For any subdivision $\{(C_i, \Lambda _{C})\}$ of  a  lattice cone $(C, \Lambda _C)$, we have
$${\rm \pres}(S(C, \Lambda _C))=\sum_i {\rm \pres}(S(C_i, \Lambda _{C_i})).
$$
So the map ${\rm \pres}\circ S$ is compatible with subdivisions.
\end{prop}

\begin {proof} It is sufficient to consider a  subdivision $\{(C_i, \Lambda _{C})\}$ of $(C,\Lambda_C)$ with smooth cones $C_i$, since any other subdivision can be further subdivided into one containing only smooth cones. It   follows from the definition of $S(C,\Lambda_C)$ that
$$S(C, \Lambda _C)=
\sum_{I\subseteq [r]} (-1)^{|I|+1} S(C_I,\Lambda_I) =\sum_i S(C_i, \Lambda_{C})+({\rm terms \ of \ p\text{-}order < \ dim}(C)),
$$
where $C_I:=\cap_{i\in I}C_i$. Also by Lemma~\mref{lem:por}
$$S(C_i, \Lambda_{C})=\sum _j T_{ij}+({\rm terms \ of \ p\text{-}order < \ dim}(C)),
$$
where $T_{ij}$ are \gname germs at zero with
${\rm \pord}(T_{ij})={\rm dim }(C)$. Thus
$$S(C, \Lambda _C)=\sum _{i,j} T_{ij}+({\rm terms \ of \ p\text{-}order < \ dim}(C)).
$$

If $C$ is strongly convex, then ${\rm \pord}(S(C, \Lambda _C))={\rm dim(C)}$, and by Proposition~\mref {pp:GenPRes},
$${\rm \pres}(S(C, \Lambda _C))=\sum_{i, j} T_{ij}(0)=\sum_{i} {\rm \pres}(S(C_i, \Lambda _C)).
$$
If $C$ is not strongly convex, then $S(C, \Lambda _C)=0$; while by Proposition \mref {pp:GenPRes}, this means $\sum _{i,j} T_{ij}=0$, that is $\sum_i {\rm \pres}(S(C_i, \Lambda _C))=0$.
So the equality in the theorem holds in either case.
\end{proof}

Note that the operator $I$ is also compatible with subdivisions \cite {GPZ2}. Thus as a consequence of Proposition~\mref{thm:ressubd}, we obtain

\begin {coro}\mlabel{cor:resS} For a lattice cone $(C, \Lambda _C)$, we have
${\rm \pres}( S(C, \Lambda _C))=I (C, \Lambda_C).
$
\mlabel{co:pres}
\end{coro}
\begin{ex} Take  $\Lambda=\Z^2\subset  \R^2$ and $C=\langle e_1,e_1+e_2\rangle$ with $(e_1,e_2)$ the canonical orthonormal basis in $\R^2$. Then $S^c(C,\Lambda_C) =\frac{1}{\left(1-e^{\e_1}\right)\left(1-e^{\e_1+\e_2}\right)}$ has p-order $2$ and p-residue
	$I(C,\Lambda_C)= \frac{1}{\e_1(\e_1+\e_2) }$.
\end{ex}

\subsection{Further perspectives: a coproduct on meromorphic germs}\label{sec:coproduct}
As an outlook for future study, we  end the paper with some observations on a coproduct on $\calm(V^\circledast\ot \C)$  derived from our Laurent theory of meromorphic germs at zero.

Let $f$ be in $\calm (V ^\circledast\otimes \C)$ with a Laurent expansion
$f=\sum\mlimits_{i=1}^N S_i+h$ supported by a properly positioned family ${\conefamilyc}$. Let $S_i=\frac{h_i}{\vec L_i^{\vec s_i}}$ and $h_0:= h$. Define
$$
\Delta_{\conefamilyc} (f):=\sum\mlimits_{i=1}^N h_i\otimes \frac 1{\vec L_i^{\vec s_i}}.
$$
Using the  compatibility  with subdivisions, we see that $\Delta_{\conefamilyc}(f)$ does not depend on the choice of Laurent expansions of $f$ and so it defines a map
$$\Delta_{\calm (V ^\circledast\otimes \C)} : \calm (V ^\circledast\otimes \C)\longrightarrow \calm (V^\circledast\otimes \C)\otimes \calm (V ^\circledast\otimes \C).$$   Furthermore, since
$$(\id\otimes \Delta_{\calm (V ^\circledast\otimes \C)})\circ \Delta_{\calm (V ^\circledast\otimes \C)} (f)=\sum_i h_i\otimes 1\otimes \frac 1{\vec L_i^{\vec s_i}}=(\Delta_{\calm (V ^\circledast\otimes \C)}\otimes \id)\circ \Delta_{\calm (V ^\circledast\otimes \C)} (f),$$
$\Delta_{\calm (V ^\circledast\otimes \C)}$ is coassociative. Thus $\Delta_{\calm (V ^\circledast\otimes \C)}$ defines a coproduct on $\calm(V^\circledast\ot V)$. This coproduct is not compatible with the multiplication $m_{\calm (V ^\circledast\otimes \C)}$ on $\calm(V^\circledast\ot V)$, so we do not have a bialgebra. However, we do have the compatibility property
\begin{equation}\label{eq:Delta}m_{\calm (V ^\circledast\otimes \C)}\circ \Delta_{\calm (V ^\circledast\otimes \C)}=\id_{\calm (V ^\circledast\otimes \C)}.
\end{equation}

We compare this coproduct with the coproduct on cones~\mcite{GPZ3}, especially in the context of renormalization a la Connes and Kreimer~\cite{CK} who regarded a renormalized map as a map defined on a coalgebra and taking values in  meromorphic functions.
So fix a linear map
$$\phi:\QQ\cc \to \calm (V^\circledast\otimes \C)$$
and assume that the inner product to construct the coproduct in $\QQ \cc$ coincides with the inner product used to define polar germs. Taking $P$ to be the projection to $\calm _+ (V^\circledast\otimes \C)$, then \abf~\cite{GPZ2,GPZ3} yields $\phi=\phi_+^{\star (-1)}\star \phi_-$ for certain linear maps $\phi_\pm: \QQ\cc \to \calm_\pm(V^\circledast\ot \C)$. These maps fit in the following diagram:
$$\xymatrix{
  \QQ \cc \ar[d]_{\Delta _{\QQ \cc}} \ar[rrr]^{\phi}
                & &&\calm (V^\circledast\otimes \C)  \ar[d]^{\Delta _{\calm (V^\circledast\otimes \C) }}  \\
  \QQ \cc\otimes \QQ \cc  \ar[rrr]^(.4){\phi_+^{\star (-1)}\otimes \phi_-}
                & &&\calm (V^\circledast\otimes \C)\otimes \calm (V^\circledast\otimes \C)             }
$$
The commutativity of this diagram for a given $\phi$ provides an alternative to the Algebraic Birkhoff Decomposition, without going through the coproduct of cones. This should be the case under suitable conditions, for example when the inner product use to construct the coproduct in $\QQ \cc$ coincides with the inner product used to define polar germs, and when $\phi$ is the exponential sum $S(C,\Lambda_C)$, in which case one would  recover the Euler-Maclaurin formula in \cite[Theorem 7.15]{BV2} and  \cite[Theorem 4.10 and Corollary 4.11]{GPZ2}.

\smallskip
\noindent
{\bf Acknowledgements}:
The authors acknowledge supports from the Natural Science Foundation of China (Grant No. 11071176, 11221101 and 11371178),  the German Research Foundation (DFG grant PA 1686/6-1) and the National Science Foundation of US (Grant No. DMS 1001855).
Part of the work was completed during visits of two of the authors at Sichuan University, Lanzhou University and Capital Normal University, to which they are very grateful.


\begin{thebibliography}{abcdsfgh}

\bibitem{Ba} A. Barvinok, Integer Points in Polyhedra, {\em Zurich
lectures in Advanced Mathematics}, European Mathematical Society, 2008.

\bibitem{BV1} N. Berline and M. Vergne, Local Euler-Maclaurin formula for polytopes, {\em Mosc. Math. J.} {\bf 7}	 (2007) 355-386.

\bibitem {BV2}  N. Berline and M. Vergne, Local asymptotic Euler-Maclaurin expansion for Riemann sums over a semi-lattice polyhedron, {ArXiv:1502.01671}

\bibitem{BrV} M. Brion and M. Vergne, Arrangement of hyperplanes. I: Rational functions and Jeffrey-Kirwan residue, Annales Scientifiques de l'\'ecole Normale Sup\'erieure, 1999, 32(5), 715-741.

\bibitem{CK} A. Connes and D. Kreimer, Hopf algebras, Renormalization and Noncommutative Geometry, {\em Comm. Math. Phys.} {\bf  199} (1988) 203-242.

\bibitem{dCP}  C. De Concini and C. Procesi,  Topics in hyperplanes, arrangements, polytopes and box-splines, Universitext Springer (2011).

\bibitem{dCP1}  C. De Concini and C. Procesi, Nested sets and Jeffrey Kirwan cycles, In: Geometric methods in
algebra and number theory, {\em Progr. Math.} {\bf 235}, 139 - 149, Birkh\"{a}user, 2005.

\bibitem{De} N.V. Dang, Complex powers of analytic functions and meromorphic renormalization in QFT arXiv:1503.00995.

\bibitem{E} I. Efrat, Valuations, Orderings, and Milnor K-theory, {\em Mathematical Surveys and Monographs}, {\bf 124}, AMS, 2006.

\bibitem{GP} S. Garoufalidis and J. Pommersheim,  Sum-integral interpolators and the Euler-Maclaurin formula for polytopes, {\em Trans. Amer. Math. Soc.} {\bf 364} (2012) 2933-2958.

\bibitem{GF} H. Grauert and K. Fritzsche, Several Complex Variables, {\em Graduate Texts in Mathematics} {\bf 38}, Springer 1976.

\bibitem{GPZ1} L. Guo, S. Paycha and B. Zhang, Conical zeta values and their double subdivision relations,  {\em Adv. in Math.} {\bf 252} (2014) 343-381.

\bibitem{GPZ2} L. Guo, S. Paycha and B. Zhang, Renormalization and the Euler-Maclaurin formula on cones, {\em Duke Math J.}, to appear,
    doi:10.1215/00127094-3715303, arXiv:1306.3420.

\bibitem{GPZ3} L. Guo, S. Paycha and B. Zhang,  Renormalized conical zeta values, to appear in a publication (Springer-Verlag) of the  Scuola Normale Superiore in Pisa,  arXiv:1602.04190.

\bibitem{GZ} L. Guo and B. Zhang, Renormalization of multiple zeta
   values {\em J. Algebra} {\bf 319} (2008) 3770-3809.

\mbibitem{Ho} M.~E.~Hoffman, {Multiple harmonic series},
 {\em Pacific J. Math.}, {\bf 152} (1992), no. 2, 275--290.

\mbibitem{JK1} L. C. Jeffrey and F. C. Kirwan, Localisation for non abelian group actions, {\em Topology} {\bf 34} (1995), 291-327.

\mbibitem{JK2} L. C. Jeffrey and F. C. Kirwan, Localization and the quantization conjecture, {\em Topology}  {\bf 36} (1997), no. 3, 647-693.

\mbibitem{JM} L. C. Jeffrey and M. Kogan, Localization theorems by symplectic cuts in {\em The breadth of symplectic
and Poisson geometry}, Progr. Math., Volume 232, 303 - 326. Birkh\"{a}user, 2005.

\bibitem{KMT} Y. Komori, K. Matsumoto and H. Tsumura, On Witten multiple zeta-functions associated with semisimple Lie algebras V, Glasgow Mathematical Journal 57 (2015) 107-130

\bibitem {La} J. Lawrence, Rational-function-valued valuations on polyhedra, DIMACS Ser. Discrete Math. Theoret. Comput. Sci.,6 (1991) 199-208.

\bibitem{MP} D. Manchon and S. Paycha, Nested sums of symbols and renormalised multiple zeta values, Int. Math. Res. Papers 2010 issue 24, 4628-4697 (2010).

\mbibitem{Ma} K. Matsumoto,
   The analytic continuation and the asymptotic behaviour of certain multiple zeta-functions I, {\it J. Number Theory}, {\bf 101} (2003), 223--243.

\bibitem{Sp1} E. Speer,  Analytic renormalization using many space-time dimensions, {\em Comm. Math. Phys.} {\bf 37} (1974), 83-92.

\bibitem{Sp2} E. Speer, Ultraviolet and infrared singularity structure of generic Feynman amplitudes,
{\em Ann. Inst. H.  Poincar\'e}, {\bf 23A} (1975),1-21.

\bibitem{SV} A. Szenes and M. Vergne, Toric reduction and a conjecture of Batyrev and Materov, {\em Invent.
math.}, 158, 453 - 495 (2004)

\bibitem{Te}T. ~ Terasoma, Rational convex cones and cyclotomic multiple zeta values,   arXiv:math/0410306v1 [math.AG]  2004.

\bibitem{Za} D.~Zagier, {Values of zeta functions and their
applications},
  First European Congress of Mathematics, Vol. II (Paris, 1992), 497--512,
{\it Progr. Math.}, {\bf{120}}, Birkh\"auser, Basel, 1994

\mbibitem{Zh} J. Zhao, Analytic continuation of multiple zeta
    functions. {\it Proc. Amer. Math. Soc.} {\bf 128} (2000),
        1275-1283.

\end{thebibliography}
\end{document}